\documentclass[a4paper, 10pt,reqno]{article}
\usepackage{amssymb, amsmath, amsfonts, amsthm, graphics,mathrsfs,mathtools,dsfont}
\usepackage{bm}
\usepackage{stmaryrd}
\usepackage[hmargin=3cm, vmargin = 3cm]{geometry}
\usepackage[T1]{fontenc}
\usepackage{lmodern}
\usepackage{colonequals}
\usepackage{tikz-cd} 
\usepackage{hyperref}
\usepackage[all]{xy}
\usepackage{authblk}

\let\emph\relax
\DeclareTextFontCommand{\emph}{\bfseries}

\theoremstyle{definition}
\newtheorem{Definition}{Definition}[section]

\newtheorem{Remark}[Definition]{Remark}

\theoremstyle{theorem}
\newtheorem{Theorem}[Definition]{Theorem}
\newtheorem{Lemma}[Definition]{Lemma}

\newtheorem{Proposition}[Definition]{Proposition}
\numberwithin{equation}{section}

\DeclareMathOperator{\dive}{div}
\DeclareMathOperator{\curle}{curl}

\newcommand{\NN}{\mathbb{N}}

\newcommand{\RR}{\mathbb{R}}

\newcommand{\cB}{\mathcal{B}}
\newcommand{\cC}{\mathcal{C}}

\newcommand{\cF}{\mathcal{F}}
\newcommand{\cG}{\mathcal{G}}

\newcommand{\cI}{\mathcal{I}}

\newcommand{\cP}{\mathcal{P}}
\newcommand{\cQ}{\mathcal{Q}}

\theoremstyle{definition}

\theoremstyle{definition}

\newcommand{\wsconverge}{\xrightharpoonup{w^*}}
\newcommand{\wconverge}{\xrightharpoonup{\,\,w\,}}
\newcommand{\sconverge}{\xrightarrow{\,\,s\,\,}}

\newcommand\abs[1]{\left|#1\right|}
\newcommand\norm[1]{\left\|#1\right\|}
\newcommand\normf[1]{\|#1\|}

\hypersetup{
hidelinks,
colorlinks=true,
linkcolor=blue,
citecolor=red,
}

\title{Global Existence of Strong Solutions and Serrin-Type Blowup Criterion for 3D Combustion Model in Bounded Domains}
\author{Jiawen Zhang\thanks{zhangjiawen@amss.ac.cn}}
\affil{\normalsize School of Mathematical Sciences,\\
University of Chinese Academy of Sciences, Bejing 100049, P.R. China}
\date{}

\begin{document}

\maketitle

\begin{abstract}
The combustion model is studied in three-dimensional (3D) smooth bounded domains with various types of boundary conditions. The global existence and uniqueness of strong solutions are obtained under the smallness of the gradient of initial velocity in some precise sense. Using the energy method with the estimates of boundary integrals, we obtain the a priori bounds of the density and velocity field. Finally, we establish the blowup criterion for the 3D combustion system.
\end{abstract}

\section{Introduction}
In this paper, we assume that $\Omega$ is a simply connected bounded domain in $\RR^3$ with smooth boundary and investigate the following system in $\Omega$,
\begin{equation}\label{equation1.1}
\begin{cases}
\rho_t+\dive (\rho u)=0,\,\,\rho\geq 0,\\
(\rho u)_t+\dive(\rho u\otimes u)-\dive[2\mu(\rho)D(u)]+\nabla \pi=0,\\
\dive u=c_0\Delta\psi(\rho),\,\,\psi(\rho):=\rho^{-1},
\end{cases}
\end{equation}
where  $u=(u_1,u_2,u_3)$, $\rho$ and $\pi$ stand for the unknown velocity field, density and pressure respectively, $c_0>0$ is a fixed constant, $\mu$ is a positive function and 
\begin{equation}\label{equation12}
\mu(s)\in C^\infty(0,\infty).
\end{equation}
The deformation tensor $D(u)$ is denoted by
\begin{equation}
D(u) = \frac{1}{2}\left[\nabla u+(\nabla u)^t\right]=\frac{1}{2}(\partial_iu_j+\partial_ju_i),\quad 1\leq i,j\leq 3.
\end{equation} 

The system is equiped with the initial data
\begin{equation}\label{equation1.4}
u(x,0)=u_0(x),\quad \rho(x,0)=\rho_0(x),\quad x\in\Omega
\end{equation}
and one of the following boundary conditions:
\begin{equation}\label{equation1.6}
n\cdot \nabla\rho=0,\quad u\cdot n=0,\,\,\curle u\times n=-B\cdot u\quad \mathrm{on}\,\,\partial\Omega\times(0,T)\tag{A}
\end{equation}
where $B=B(x)$ is a smooth positive semi-definite matrix, or
\begin{equation}\label{equation1.7}
n\cdot \nabla\rho=0,\quad u=0\quad \mathrm{on}\,\,\partial\Omega\times(0,T),\tag{B}
\end{equation}

Combustion model is the low Mach number limit of the fully compressible Navier-Stokes equations, see \cite{lions1}, and it is tightly linked with the non-homogeneous incompressible Navier-Stokes equations (taking $c_0=0$) and the homogeneous one (taking $\rho$ be a constant). There are lots of works studying the combustion model \eqref{equation1.1} and the problems associated with it. The study of the system \eqref{equation1.1}, which has been introduced by A. Majda \cite{majda}, can date back to the 1980s. P. Embid \cite{embid} has proved the local-in-time well-posedness for classical solutions of the system \eqref{equation1.1} with the periodic boundary condition. Also, the local well-posedness was considered by H. B. da Veiga \cite{daveiga} with $\eqref{equation1.1}_3$ replaced by Fick's law $\psi(\rho)=\log \rho$. Danchin-Liao \cite{danchin}  established the local well-posedness in critical homogeneous Besov spaces under some smallness assumptions and that in non-homogeneous Besov space for arbitrarily large data.

For the global-in-time existence of weak and strong solutions of \eqref{equation1.1} and relative problems, P. Secchi \cite{secchi} proved that there exists a unique global strong solution in the two-dimensional domain providing the diffusion coefficient $c_0$ is small enough. They also considered the limiting behavior of the solutions when $c_0\to 0^+$ for 2D and 3D case and the convergence towards the corresponding solutions of non-homogeneous incompressible Navier-Stokes equations. Another remarkable work comes from P. Lions \cite{lions2} where he has shown the global existence of weak solutions only under a small perturbation of a constant density without any restriction on the initial velocity. However, in \cite{lions2}, he only gives the proof for $\RR^2$ and periodic case. Also in \cite{danchin}, Danchin-Liao proved the existence of solutions in critical homogeneous Besov spaces provided the initial density is closed to a constant and the initial velocity is small enough. For large initial data, Bresch-Essoufi-Sy \cite{bresch2} showed the global existence of the weak solutions for the combustion model in dimensions 2 and 3 by taking $\mu(\rho)$ be a specific function $\frac{c_0}{2}\log \rho$ and then, in \cite{bresch1}, Bresch-Giovangigli-Zatorska relaxed the restriction on $\mu(\rho)$ by renormalizing the mass equation. Recently, W. Tan \cite{tan} proved the global existence of the weak and strong solutions for the system \eqref{equation1.1} with general coefficient $\mu(\rho)$ in $\eqref{equation1.1}_2$ and $\psi(\rho)$ in $\eqref{equation1.1}_3$ provided $\norm{\nabla\rho}_{L^2}$ is small enough.

Another relative model to the system \eqref{equation1.1} is the so-called Kazhikhov-Smagulov type model, see \eqref{equation1.18}. In \cite{caixiaoyun,aaa}, Cai-Liao-Sun established the global-in-time existence of strong solutions to the initial-boundary value problem of a 2D Kazhikhov-Smagulov type model for incompressible non-homogeneous fluids with mass diffusion for the arbitrary size of initial data. For other works on the classical Kazhikhov-Smagulov's model, we refer the reader to \cite{antontsev,beirao}.

If the diffusion coefficient $c_0$ tends to zero, \eqref{equation1.1}	may reduce to the general non-homogeneous incompressible Navier-Stokes equations. There are also plenty of works studying it with the general viscosity coefficient $\mu(\rho)$, we refer the reader to \cite{abidi,cai2,cho,he2021,jun,huang,lions1} and the references therein.

In the final part of this paper we focus on the mechanism of blowup and the structure of possible singularities of strong solutions to the Navier-Stokes system. The blowup criterion on the Leray-Hopf weak solutions to the 3D incompressible homogeneous Navier-Stokes equations was first given by J. Serrin \cite{serrin1962}, that is, if a weak solution $u$ satisfies
\begin{equation}\label{serrin1.5}
u\in L^s(0,T;L^r),\quad \frac{2}{s}+\frac{3}{r}\leq 1,\quad 3<r\leq \infty,
\end{equation}
then it is regular. Later, He-Xin \cite{he2005} showed that the Serrin's criterion \eqref{serrin1.5} still holds even in the case of the incompressible MHD equations. For non-homogeneous incompressible Navier–Stokes equations, H. Kim \cite{kim2006} has shown that if $(\rho,u)$ blows up at $T^*$, then
\begin{equation}
\lim_{t\to T^*}\norm{u}_{L^s(0,T;L^r_w)}=\infty\quad \text{for all}\quad\frac{2}{s}+\frac{3}{r}\leq 1,\quad 3<r\leq \infty.
\end{equation}
In recent works, X. Zhong \cite{zhong2017} obtained a blowup criterion \eqref{serrin1.5} to the non-homogeneous incompressible heat conducting Navier–Stokes flows in bounded domain of $\RR^3$. For the compressible fluids, we refer reader to \cite{huang2013serrin,HLX,xu2012blow} and references therein.

However, the theory for the 3D combustion model with the general viscosity coefficient in the bounded domain is still blank. Therefore, our goal is obtaining the global existence of strong solutions with small initial data and extending the Serrin's blow-up criterion to \eqref{equation1.1}.

Before stating the main theorem, let us explain some notation and conventions used throughout the paper. First, we define the strong solutions as follows.
\begin{Definition}
$(\rho,u,\pi)$ is called a strong solution of \eqref{equation1.1} on $\Omega\times(0,T)$, if \eqref{equation1.1} holds almost everywhere in $\Omega\times(0,T)$ such that
\begin{equation}\label{strong}
\begin{cases}
\alpha\leq \rho\leq \beta,\\
\rho\in C([0,T];H^2)\cap L^2(0,T;H^3),\rho_t\in C([0,T];L^2)\cap L^2(0,T;L^2),\\
u\in C([0,T];H^1)\cap L^2(0,T;H^2),u_t\in L^2(0,T;L^2),\\
\pi\in L^2(0,T;H^1).
\end{cases}
\end{equation}
In particular, if \eqref{strong} holds for all $T\in (0,\infty)$, we call $(\rho,u,\pi)$ the global strong one.
\end{Definition}

For $1 \leq p \leq \infty$ and integer numbers $k \geq 1$, the standard Sobolev spaces and other functional spaces are defined as follows:
$$\begin{cases}
L^p=L^p(\Omega), \quad W^{k, p}=W^{k, p}(\Omega), \quad H^k=W^{k, 2}, \\
W_0^{k,p}=\overline{C_0^\infty}\,\,\text{closure in the norm of } W^{k,p},\\ 
\|\cdot\|_{B_1 \cap B_2}=\|\cdot\|_{B_1}+\|\cdot\|_{B_2}, \text { for two Banach spaces } B_1 \text { and } B_2, \\ 
H^1_\omega:=\left\{u\in H^1: u\cdot n=0,\,\,\curle u\times n=-B\cdot u\,\,\mathrm{on}\,\,\partial\Omega\right\}.
\end{cases}$$
Next, we set
$$
\int f d x := \int_{\Omega} f d x,\quad\int_{\partial} f := \int_{\partial\Omega} f d S
$$
and 
$$f_\Omega:=\frac{1}{|\Omega|}\int f$$
which is the average of a function $f$ over $\Omega$.

The weak, weak* and strong convergence of a sequence $\{f^n\}$ are respectively denoted by
\begin{equation*}
f^n\wconverge f,\quad f^n\wsconverge f,\quad f^n\sconverge f.
\end{equation*}

Finally, for two $3 \times 3$ matrices $A=\left\{a_{i j}\right\}, B=\left\{b_{i j}\right\}$, the symbol $A: B$ represents the trace of $A B$, that is,
$$A:B:=\mathrm{tr}(AB)=\sum_{i,j=1}^3a_{ij}b_{ji}.$$

Now, we give our main theorems. The first theorem concerns with the global existence of strong solutions for \eqref{equation1.1} when $\Omega$ is a bounded domain.

\begin{Theorem}\label{Theorem1.1}
Suppose that $\Omega\subset\RR^3$ is a simply connected bounded domain with smooth boundary and $(\rho_0,u_0)$ satisfies 
\begin{equation}\label{equation1.8}
0<\alpha\leq \rho_0\leq\beta<\infty,\quad x\in \Omega,
\end{equation}
the compatibility condition
\begin{equation}\label{equation1.3}
\begin{cases}
\dive u_0=c_0\Delta \rho_0^{-1},&x\in \Omega\\
u_0\cdot n=c_0n\cdot\nabla \rho_0^{-1},&x\in \partial\Omega\\
\end{cases}
\end{equation}
and $u_0\in H^1_\omega$, if $u$ satisfies the boundary condition \eqref{equation1.6}; $u_0\in H^1_0$, if $u$ satisfies the boundary condition \eqref{equation1.7}. 

Then there exists a positive constant $\delta$ depending only on $\Omega$, $c_0$, $\alpha$ and $\beta$ such that if 
\begin{equation}
\norm{\nabla u_0}_{L^2}\leq \delta
\end{equation}
and $\pi$ satisfies the normalized condition
\begin{equation}\label{normalized}
\int\pi=0,
\end{equation}
the system \eqref{equation1.1}--\eqref{equation1.4}, \eqref{equation1.6} or \eqref{equation1.7} admits a unique global strong solution $(\rho,u,\pi)$.
\end{Theorem}

Next, we give the Serrin-type blowup criterion.
\begin{Theorem}\label{Theorem1.3}
If $(\rho, u,\pi)$ is a local strong solution on $\Omega\times(0,T^*)$ and $T^*<\infty$ is the maximal time of existence, then
\begin{equation}\label{serrin}
\lim_{T\to T^*}\norm{u}_{L^r(0,T;L^s)}=\infty,
\end{equation}
where $r$ and $s$ satisfy the relation
\begin{equation}
\frac{2}{s}+\frac{3}{r}\leq 1,\quad 3<r\leq \infty.
\end{equation}
\end{Theorem}

\begin{Remark}
Our main theorems holds for all function $\mu(s)>0$ satisfying \eqref{equation12} even if $\mu(s)\to\infty$ as $s\to 0^+$ under the smallness assumption on $\norm{\nabla u_0}_{L^2}$. Theorem \ref{Theorem1.1} is the first result giving the existence of strong solutions for \eqref{equation1.1} with general viscosity coefficient in an arbitrary 3D bounded domain. Theorem \ref{Theorem1.3} is parallel to the classical Serrin's condition for 3D non-homogeneous Navier-Stokes equations.
\end{Remark}

\begin{Remark}
Comparing with the work of  \cite{huang,zhang2015} where they obtain the global strong solutions for non-homogeneous incompressible Navier-Stokes equations with density-depended viscosity coefficient $\mu(\rho)$ and the Dirichlet boundary conditions, our result can be seen as an extension from the divergence-free velocity field $u$, $\dive u=0$, to non-divergence-free one, that is, $\eqref{equation1.1}_3$.
\end{Remark}

\begin{Remark}
In our proof of the theorem, we only need $\psi(s)\in C^3(0,\infty)$, thus, more general $\psi(\rho)$ can also be considered under the same assumptions.
\end{Remark}

\begin{Remark}
From the hypothesis of Theorem \ref{Theorem1.1}, one may notice that we do not impose any information about the regularity of $\rho_0$ (except for the size restriction \eqref{equation1.8}). This is mainly because of the compatibility condition \eqref{equation1.3}. Indeed, for example, if $u_0\in H^1_\omega$, one can solve the following elliptic problem
\begin{equation*}
\begin{cases}
c_0\Delta\rho_0^{-1}=\dive u_0,&x\in \Omega,\\
n\cdot \nabla\rho_0^{-1}=0,&x\in \partial\Omega,
\end{cases}
\end{equation*}
from which the regularity of $\rho_0$ is completely determinded by that of $u_0$. More precisely, we have, for all $1<p\leq 6$,
\begin{equation}\label{1.16}
\begin{cases}
\norm{\nabla\rho_0}_{L^p}\leq C(p)\norm{u_0}_{L^p},\\
\norm{\nabla\rho_0}_{H^1}\leq C\norm{\nabla u_0}_{L^2}.
\end{cases}
\end{equation}
\end{Remark}

Now, we give some comments about the analysis throughout the whole paper. Generally speaking, in order to overcome the non-divergence-free of $u$, our proof for Theorem \ref{Theorem1.1} is based on two types of decomposition. For the first case, that is, $u$ satisfying the boundary condition \eqref{equation1.6}, we may write in view of $\eqref{equation1.1}_3$
\begin{equation}\label{equation1.17}
v = u-c_0\nabla\rho^{-1},
\end{equation}
Consequently, using \eqref{equation1.17}, the original system \eqref{equation1.1} can be changed into the following Kazhikohlv-Samgulov type model,
\begin{equation}\label{equation1.18}
\begin{cases}
\rho_t +v\cdot\nabla\rho + c_0\rho^{-2}\abs{\nabla\rho}^2 - c_0\rho^{-1}\Delta\rho=0,\\
\\
\begin{cases}
(\rho v)_t +\dive(\rho  v\otimes v) - \dive{[2\mu(\rho)D(v)]}+ \nabla \pi_1=c_0\dive{\left[2\mu(\rho)\nabla^2\rho^{-1}\right]}\\
- c_0 \dive{\left(\rho v\otimes\nabla\rho^{-1}\right)} -\dive\left(c_0\rho\nabla\rho^{-1}\otimes v\right)-c_0^2 \dive{\left(\rho \nabla\rho^{-1}\otimes\nabla\rho^{-1}\right)},
\end{cases}\\
\\
\dive{v} = 0,
\end{cases}
\end{equation}
where $\pi_1=\pi-c_0(\log\rho)_t$ is a modified pressure. Then, one can find that the mass equation $\eqref{equation1.1}_1$ becomes a parabolic type one, which provides us some high regularity properites for $\rho$, and, on the other hand, $v$ is divergence-free, which allows us to use some ``standard'' treatments of the classical incompressible Navier-Stokes equations. Thus, in Section \ref{section3}, we will mainly discuss the system \eqref{equation1.18} and try to derive the a priori esitmates of $(\rho,v)$. 

So, here, we give an explanation about the definition of $v_0$, the initial value of $v$, and the boundary condition related to $v$. Since we have the compatibility condition \eqref{equation1.3} from which we can find a unique function $v_0$ defined by
\begin{equation}\label{1.19}
v_0:=u_0-c_0\nabla\rho_0^{-1}.
\end{equation}
Then, we may impose $v_0$ as the initial value of $v$. Of course, in view of the estimates \eqref{1.16}, $v_0$ is also controlled by $u_0$, that is,
\begin{equation}\label{1.17}
\begin{cases}
\norm{v_0}_{L^p}\leq C(p)\norm{u_0}_{L^p},\\
\norm{\nabla v_0}_{L^2}\leq C\norm{\nabla u_0}_{L^2}.
\end{cases}
\end{equation}
For the boundary condition, if $u$ satisfies the condition \eqref{equation1.6}, applying $\curle$ on \eqref{equation1.17} implies that $v$ satisfies
\begin{equation}\label{A'}
\curle v\times n=-B\cdot(v+c_0\nabla\rho^{-1})\quad\mathrm{on}\,\,\partial\Omega\times(0,T).\tag{A'}
\end{equation}
In this case, we would call $(\rho,v)$ or $v$ satisfying the condition \eqref{A'}. In addition, from \eqref{1.19}, we can obtain the compatibility condition corresponding with $(\rho_0,v_0)$, that is,
\begin{equation}
\begin{cases}
\dive v_0=0,&x\in \Omega\\
v_0\cdot n=0,\curle v_0\times n=-B\cdot(v_0+c_0\nabla\rho_0^{-1}),&x\in \partial\Omega\\
\end{cases}
\end{equation}
provided $u_0\in H^1_\omega$. To sum up, our sketches of the proof is given by
\begin{equation*}
(\rho_0,u_0)\xRightarrow{\eqref{1.19}}(\rho_0,v_0)\implies \mathrm{existence\,\,of}\,\,(\rho,v)\xRightarrow{\eqref{equation1.17}}\mathrm{existence\,\,of}\,\,(\rho,u).
\end{equation*}

Another difficulty in this situation comes from the boundary integrals. To overcome it, we mainly adapt the idea from Cai-Li \cite{cai}. Since $v \cdot n=0$ on $\partial \Omega$, we have
$$
v=v^{\perp} \times n \quad\text {on } \partial \Omega,
$$
where $v^\perp= -v\times n$. Then, for $f\in H^1$,
$$\abs{\int_\partial v\cdot\nabla f}=\abs{\int_\partial v^\perp\times n\cdot \nabla f}=\abs{\int \curle v^\perp\cdot \nabla f}\leq C\norm{v}_{H^1}\norm{\nabla f}_{L^2},$$
which is clearly has advantages over using the trace inequality, since the latter needs $f\in H^2$.

For $u$ satisfying \eqref{equation1.7}, the situation is somewhat different, since, in every case that follows, $v$ satisfies the non-homogeneous Dirichlet boundary conditions, that is,
\begin{equation}
v=-c_0\nabla\rho^{-1}\quad\mathrm{on}\,\,\partial\Omega\times (0,T).
\end{equation}
Such condition may bring too much high order derivatives so that the boundary integrals are no longer controllable, especially when we treat the energy estimates for $v$. Therefore, we shall apply another type of decomposition whose idea comes from Lemma \ref{llemma2.2} (see Section \ref{section2}). From which, one can find a function $Q=\cB[c_0\Delta\rho^{-1}]$, where $\cB$ is the Bogovski\v i operator. As a consequence, $u$ will be splitted into
\begin{equation}\label{1.21}
u=w+Q.
\end{equation}
and, hence, one can hope to get the energy estimates for the system \eqref{equation1.1}

The advantage of above decomposition is obvious: on the one hand, from Lemma \ref{llemma2.2}, $Q$ is ``almost'' $\nabla\rho$, in other words, for all $1<p<\infty$, $Q$ has the following bounds
\begin{equation}\label{1.23}
\begin{cases}
\norm{Q}_{L^p}\leq C\norm{\nabla\rho}_{L^p},\\
\norm{Q}_{H^1}\leq C\left(\norm{\Delta\rho}_{L^2}+\norm{\nabla\rho}_{L^{3}}\norm{\nabla\rho}_{L^{6}}\right),\\
\norm{Q_t}_{L^p}\leq C\left(\norm{\nabla\rho_t}_{L^p}+\norm{|\rho_t||\nabla\rho|}_{L^p}\right);
\end{cases}
\end{equation}
on the other hand, it is easy to check that $w$ has a vanished boundary, which will not generate any bounary term when applying the energy estimates. Therefore, the strategy of the proof can be concluded as follows
\begin{equation*}
\begin{aligned}
(\rho_0,u_0)\xRightarrow[\eqref{1.23}]{\eqref{1.21}}\mathrm{estimates\,\,for}\,\,(\rho,u)\implies\cdots 
\end{aligned}
\end{equation*}

At last, to prove Theorem \ref{Theorem1.3}, we mainly adapt proofs mentioned above with a slight change. We first let \eqref{serrin} be false, that is,
\begin{equation}\label{serrin'}
\lim_{t\to T^*}\norm{u}_{L^s(0,T;L^r)}\leq M_0<\infty,
\end{equation}
then following the proof of Theorem \ref{Theorem1.1}, one may obtain the bounds for $(\rho,u,\pi)$ satisfying \eqref{strong}, which will give the contradictory ot the maximality of $T^*$. However, when it comes to the higher order estimates of $(\rho,v)$ (or $(\rho,u)$), one has to control 
$$\normf{|\nabla\rho|^3}_{L^2}=\norm{\nabla\rho}_{L^6}^3,$$
due to the nonlinear terms
$$\rho^{-2}\abs{\nabla\rho}^2,\quad c_0^2 \dive{\left(\rho \nabla\rho^{-1}\otimes\nabla\rho^{-1}\right)}$$
in $\eqref{equation1.18}_1$ and $\eqref{equation1.18}_2$, which is failed to be bounded by the Serrin's condition \eqref{serrin'}. To overcome it, we change $\eqref{equation1.18}_1$ into
\begin{equation}
\rho_t+v\cdot \nabla\rho-c_0\Delta\log\rho=0,
\end{equation}
which pushes us to estimate $\log\rho$ thanks to the pure transport constructure $\rho_t+v\cdot \nabla\rho$ and the disspation term $-\Delta\log\rho$, see Section \ref{section4} for details.

The rest of this paper is organized as follows. In Section \ref{section2}, we give some elementary results
which will be used in later. Section \ref{section3} is devoted to the a priori estimates for system \eqref{equation1.1} and the proof for Theorem \ref{Theorem1.1}. Finally, in Section \ref{section4}, we will give the proof of Theroem \ref{Theorem1.3}.

\section{Preliminaries}\label{section2}

First, we give the following local existence result for system \eqref{equation1.1}. We have already proved this for 2D case in our previous work \cite{zjw} and the 3D one can be established step by step only after some minor adaptions.
\begin{Lemma}\label{local}
Assume that $(\rho_0,u_0)$ satisfies the same conditions as in Theorem \ref{Theorem1.1} and $\Omega\subset \RR^3$ is a simply connected bounded domain with smooth boundary. Let $\pi$ saitisfies the condition \eqref{normalized}. Then there exists a positive time $T_1$ depending on $\Omega$, $c_0$, $\alpha$, $\beta$ and $\norm{u_0}_{H^1}$ so that the problem \eqref{equation1.1}--\eqref{equation1.4}, \eqref{equation1.6}admits an unique strong solution $(\rho, u,\pi)$ on $\Omega\times(0,T_1)$. 

Moreover, if $\mu(\rho)$ is a positive constant, then the above result also holds for the condition \eqref{equation1.7}. 
\end{Lemma}

\begin{Remark}\label{local2}
Even if we restrict $\mu(\rho)=\mu$ a positive constant in the case \eqref{equation1.7}, the existence result can be extended to $\mu(\rho)=\mu(\rho_\epsilon)$ (see \cite{zjw} for details), where 
\begin{equation*}
\rho_\epsilon\in C^\infty(\overline\Omega),\quad \alpha\leq \rho_\epsilon\leq \beta,\quad \rho_\epsilon\sconverge \rho\quad\text{in }W^{k,p}\text{ for all }\rho\in W^{k,p},\,k\in \NN,\,1\leq p<\infty.
\end{equation*}
This extension will help us to fill the gap between the existence of local strong solutions and that of global one when $(\rho,u)$ satisfies the condition \eqref{equation1.7}.
\end{Remark}

Next, we give the well-known Gagliardo-Nirenberg's inequalities which will be frequently used later.
\begin{Lemma}[Gagliardo-Nirenberg \cite{leoni,nirenberg}] \label{Lemma221}
Assume that $\Omega$ is a bounded domain in $\mathbb{R}^3$ with smooth boundary. Then there exist generic constants $C$ and $C_1$ which depend only on $p$ and $\Omega$ such that, for all $p \in[2,6]$ and $f \in H^1$,
\begin{gather*}
\|f\|_{L^p(\Omega)} \leq C\|f\|_{L^2}^{\frac{6-p}{2 p}}\|\nabla f\|_{L^2}^{\frac{3 p-6}{2 p}}+C_1\|f\|_{L^2}.
\end{gather*}
Moreover, if either $\left.f \cdot n\right|_{\partial \Omega}=0$ or $f_\Omega=0$, we can choose $C_1=0$.
\end{Lemma}

The next two lemmas can be found in \cite{aramaki,von}.
\begin{Lemma}\label{lemma22}
Let $\Omega$ be a bounded simply connected domain in $\mathbb{R}^3$ with smooth boundary. Assume that $k\geq 0$ is an integer and $1<p<\infty$. Then for all $u\in W^{{k+1},p}$ with $u\cdot n=0$ on $\partial \Omega$, there exists a positive constant $C=C(k,p,\Omega)$ such that
\begin{equation*}
\norm{u}_{W^{{k+1},p}}\leq C\left(\norm{\dive{u}}_{ W^{{k},p}}+\norm{\curle{u}}_{ W^{{k},p}}\right).
\end{equation*}
\end{Lemma}

\begin{Lemma}\label{lemma23}
Suppose that $\Omega$ is a bounded simply connected domain in $\mathbb{R}^3$ smooth boundary. Let $k \geq 0$ be an integer, $1<p<\infty$. Then for $u\in W^{k+1, p}$ with $u \times n=0$ on $\partial \Omega$, there exists a constant $C=C(k, p, \Omega)$ such that
$$
\|u\|_{W^{k+1, p}} \leq C\left(\|\operatorname{div} u\|_{W^{k, p}}+\|\operatorname{curl} u\|_{W^{k, p}}+\|u\|_{L^p}\right).
$$
\end{Lemma}

Next, consider the problem
\begin{equation}\label{laplace}
\begin{cases}
\dive u=f, & x \in \Omega, \\ 
u =\Phi, & x \in \partial \Omega,
\end{cases}
\end{equation}
where $\Omega$ is a bounded smooth domain in $\mathbb{R}^3$.  We have the following standard estimates, which will be used to eliminate the non-homogeneity of equations.

\begin{Lemma}[\cite{galdi}, Theorem III.3.3]\label{llemma2.2}
Suppose that $\Phi\cdot n=0$ on $\partial\Omega$ and $f_\Omega =0$. Then,
\begin{enumerate}
\item[1)]
If $\Phi=0$, there exists a bounded linear operator $\mathcal{B}=\left[\mathcal{B}_1, \mathcal{B}_2,\mathcal{B}_3\right]$,
\begin{equation*}
\mathcal{B}: \{f\in L^{p}:f_\Omega =0\} \mapsto \left[W_0^{1,p}\right]^3
\end{equation*}
such that
\begin{equation*}
\|\mathcal{B}[f]\|_{W^{1, p}} \leq C(p)\|f\|_{L^{p}},
\end{equation*}
for all $p \in(1, \infty)$, and the function $Q=\mathcal{B}[f]$ solves the problem \eqref{laplace}. Moreover, if $f=\dive g$ with a certain $g \in L^r,\left.g \cdot n\right|_{\partial \Omega}=0$, then for any $r \in(1, \infty)$
\begin{equation*}
\|\mathcal{B}[f]\|_{L^r} \leq C(r)\|g\|_{L^r} .
\end{equation*}
$\cB$ is so-called the Bogovski\v i operator.
\item[2)]
If $f=0$, there exists a bounded linear operator $\cC=[\cC_1,\cC_2,\cC_3]$,
$$\cC: \{\Phi: \Phi\cdot n|_{\partial\Omega}=0,\,\,\dive\Phi\in L^p\}\mapsto  \left[W^{1,p}\right]^3$$
such that
$$\norm{\cC[\Phi]}_{W^{1,p}}\leq C(p)\norm{\dive \Phi}_{L^p},$$
for all $p\in (1,\infty)$ and the function $R=\cC[\Phi]$ sovles the problem \eqref{laplace}.
\end{enumerate}
\end{Lemma}

The next two lemmas about the estimates of Stokes system are important to the higher order estimates of $v$. 
\begin{Lemma}\label{lemma2.3}
Let $\Omega$ be a bounded simply connnected domain in $\mathbb{R}^3$ with smooth boundary and $(u,p)$ satisfy the following Stokes equations
\begin{equation}\label{equation2.1}
\begin{cases}
-\Delta u+\nabla p=F, &x\in\Omega,\\
\dive u=0, &x\in\Omega,
\end{cases}
\end{equation}
where $p$ is normalized by the condition $\int p=0$ and $F\in L^2$. Then, we have the following conclusions:
\begin{enumerate}
\item[(1)] 
If $u$ satisfies the boundary condition $u\cdot n=0,\,\curle u\times n=\Phi$ on $\partial\Omega$, where $\Phi\in H^1$ is a function defined on $\Omega$. Then there exists a positive constant $C$ depending only on $\Omega$ such that
\begin{equation}\label{equation2.2}
\norm{u}_{H^2} +\norm{p}_{H^1}\leq C(\normf{F}_{L^2}+\norm{\Phi}_{H^1}).
\end{equation}
\item[(2)] 
If $u$ satisfies the boundary condition $u=\Phi$ on $\partial\Omega$, where $\Phi\in H^2$ is a function defined on $\Omega$. Then there exists a positive constant $C$ depending only on $\Omega$ such that
\begin{equation}\label{equation2.2'}
\norm{u}_{H^2} +\norm{p}_{H^1}\leq C(\normf{F}_{L^2}+\norm{\Phi}_{H^2}).
\end{equation}
\end{enumerate}
\end{Lemma}
\begin{proof}
We only give the proof for $(1)$, since $(2)$ can be found in \cite{galdi}, Chapter IV. Multiplying $u$ on both side of $\eqref{equation2.1}_1$ and integrating by parts, one has
\begin{equation*}
\int |\curle u|^2 =\int_\partial \Phi\cdot u +\int F\cdot u,
\end{equation*}
which, using Lemma \ref{lemma22} and trace inequality, implies that
\begin{equation}\label{equation23}
\norm{u}_{H^1}\leq C\left(\norm{F}_{L^2}+\norm{\Phi}_{H^1}\right).
\end{equation}
Then, $\nabla p\in H^{-1}$ and, using the condition $\int p=0$, we have 
\begin{equation}\label{equation24}
\norm{p}_{L^2}\leq C\left(\norm{F}_{L^2}+\norm{u}_{H^1}\right).
\end{equation}

Next, applying $\curle$ on $\eqref{equation2.1}_1$ leads to the following Laplace equations,
\begin{equation*}
-\Delta\curle u=\curle F.
\end{equation*}
Then, multiplying $\curle u-\Phi^\perp$ and integrating over $\Omega$ gives
\begin{align*}
&\int |\curle \curle u|^2-\int \curle \curle u\cdot \curle\Phi^\perp+\int_\partial (n\times \curle \curle u)\cdot \left(\curle u-\Phi^\perp\right)\\
&=\int F\cdot \left(\curle\curle u-\curle\Phi^\perp\right)+\int_\partial (n\times F)\cdot \left(\curle u-\Phi^\perp\right),
\end{align*}
that is, using the indentity $a\cdot (b\times c)=b\cdot (c\times a)=c\cdot (a\times b)$,
\begin{align*}
\int |\curle \curle u|^2-\int \curle \curle u\cdot \curle\Phi^\perp=\int F\cdot \left(\curle\curle u-\curle\Phi^\perp\right),
\end{align*}
which imlplies that
\begin{equation*}
\norm{\curle\curle u}_{L^2}\leq C\left(\norm{F}_{L^2}+\norm{\Phi}_{H^1}\right).
\end{equation*}
It follows from Lemma \ref{lemma22}--\ref{lemma23} and \eqref{equation23} that 
\begin{equation}\label{equation27}
\norm{u}_{H^2}\leq C\left(\norm{F}_{L^2}+\norm{\Phi}_{H^1}+\norm{u}_{L^2}\right).
\end{equation}
Because of the uniqueness of the Stokes system, one can eliminate the $L^2$-norm of $u$ on the right-hand side of \eqref{equation27}. On the other hand, of course, we have 
$$\norm{p}_{H^1}\leq C\norm{\nabla p}_{L^2}\leq C(\norm{\Delta u}_{L^2}+\norm{F}_{L^2}).$$
Thus, alonging with \eqref{equation27}, we complete the proof.
\end{proof}

\begin{Lemma}\label{lemma26}
Let $\Omega$ be a bounded simply connnected domain in $\mathbb{R}^3$ with smooth boundary. Let $(u,p)$ be a strong solution of the following Stokes type system,
\begin{equation}\label{equation28}
\begin{cases}
-\dive[2\mu(\rho)D(u)]+\nabla p=F, &x\in\Omega,\\
\dive u=0, &x\in\Omega,
\end{cases}
\end{equation}
where $p$ is normalized by the condition $\int p=0$, $F\in L^2$ and 
$$0<\underline\mu \leq \mu(\rho)\leq \overline\mu<\infty,\,\,\,\,\nabla\mu(\rho)\in L^r,\quad r\in (3,\infty]$$
Then, we have the following results:
\begin{enumerate}
\item[(1)]
If $u$ satisfies the boundary condition $u\cdot n=0,\,\curle u\times n=\Phi$ on $\partial\Omega$, where $\Phi\in H^1$ is a function defined on $\Omega$. Then there exists a positive constant $C$ depending only on $\underline\mu$, $\overline\mu$ and $\Omega$ such that
\begin{equation*}
\norm{u}_{H^2}+\norm{p}_{H^1}\leq C\left[\norm{\nabla\mu(\rho)}^{\frac{r}{r-3}}_{L^r}\norm{\nabla u}_{L^2}+\left(1+\norm{\nabla\mu(\rho)}^{\frac{r}{r-3}}_{L^r}\right)\left(\norm{F}_{L^2}+\norm{\Phi}_{H^1}\right)\right].
\end{equation*}
\item[(2)]
If $u$ satisfies the boundary condition $u=\Phi$ on $\partial\Omega$, where $\Phi\in H^2$ is a function defined on $\Omega$. Then there exists a positive constant $C$ depending only on $\underline\mu$, $\overline\mu$ and $\Omega$ such that
\begin{equation*}
\norm{u}_{H^2}+\norm{p}_{H^1}\leq C\left[\norm{\nabla\mu(\rho)}^{\frac{r}{r-3}}_{L^r}\norm{\nabla u}_{L^2}+\left(1+\norm{\nabla\mu(\rho)}^{\frac{r}{r-3}}_{L^r}\right)\left(\norm{F}_{L^2}+\norm{\Phi}_{H^2}\right)\right].
\end{equation*}
\end{enumerate}
\end{Lemma}

\begin{proof}
We still only give the proof of $(1)$, since $(2)$ can be checked in a similar way. Using Lemma \ref{llemma2.2}, we can find a function $R=\cC[\Phi]$ such that $\dive R=0$ and $R|_{\partial\Omega}=\Phi$. Then, we rewrite $\eqref{equation28}_1$ as 
\begin{equation}\label{equation29}
-\operatorname{div}[2 \mu(\rho) D(u-R)]+\nabla p=F+\operatorname{div}[2 \mu(\rho) D(R)].
\end{equation}
Multiplying $u-R$ on both sides of \eqref{equation29}, integrating by parts like in Lemma \ref{lemma2.3} and using the control $\norm{R}_{H^1}\leq C\norm{\dive \Phi}_{L^2}$, one has
\begin{equation}
\norm{u}_{H^1}+\norm{p}_{L^2}\leq C\left(\norm{F}_{L^2}+\norm{\Phi}_{H^1}\right).
\end{equation}
Next, converting $\eqref{equation28}_1$ into the form
\begin{equation}
-\Delta u+\nabla\left[\frac{p}{\mu(\rho)}\right]=\frac{F}{\mu(\rho)}+\frac{2\nabla\mu(\rho)\cdot D(u)}{\mu(\rho)} -\frac{p\nabla\mu(\rho)}{\mu(\rho)^2},
\end{equation}
then using $\eqref{equation2.2}_2$ and Poincar\'e's inequality, we have
\begin{equation*}
\begin{aligned}
\norm{u}_{H^2}+\norm{p}_{H^1}&\leq C\left(\norm{u}_{H^2}+\norm{\nabla\frac{p}{\mu(\rho)}}_{L^2}+\norm{\frac{\nabla\mu(\rho)\cdot D(u)}{\mu(\rho)}}_{L^2}+\norm{\frac{p\nabla\mu(\rho)}{\mu(\rho)^2}}_{L^2}\right)\\
&\leq C(\normf{F}_{L^2}+\norm{\Phi}_{H^1}+\norm{\nabla\mu(\rho)}_{L^r}\norm{\nabla u}_{L^{\frac{2r}{r-2}}}+\norm{\nabla\mu(\rho)}_{L^r}\norm{p}_{L^{\frac{2r}{r-2}}})\\
&\leq C\left(\normf{F}_{L^2}+\norm{\Phi}_{H^1}+\norm{\nabla\mu(\rho)}^{\frac{r}{r-3}}_{L^r}\norm{\nabla u}_{L^2}+\norm{\nabla\mu(\rho)}^{\frac{r}{r-3}}_{L^r}\norm{p}_{L^2}\right)\\
&\quad+\frac{1}{2}\left(\norm{u}_{H^2}+\norm{p}_{H^1}\right),
\end{aligned}
\end{equation*}
which implies that
\begin{equation*}
\begin{aligned}
\norm{u}_{H^2}+\norm{p}_{H^1}&\leq C\left(\normf{F}_{L^2}+\norm{\Phi}_{H^1}+\norm{\nabla\mu(\rho)}^{\frac{r}{r-3}}_{L^r}\norm{\nabla u}_{L^2}+\norm{\nabla\mu(\rho)}^{\frac{r}{r-3}}_{L^r}\norm{p}_{L^2}\right)\\
&\leq C\left[\normf{F}_{L^2}+\norm{\Phi}_{H^1}+\norm{\nabla\mu(\rho)}^{\frac{r}{r-3}}_{L^r}\norm{\nabla u}_{L^2}+\norm{\nabla\mu(\rho)}^{\frac{r}{r-3}}_{L^r}\left(\norm{F}_{L^2}+\norm{\Phi}_{H^1}\right)\right].
\end{aligned}
\end{equation*}
This completes the proof.
\end{proof}

Next, we consider the H$\mathrm{\ddot{o}}$lder continuity of $\rho$ and the non-divergence type Stokes model. 
\begin{Lemma}[\cite{aaa,ladyzhenskaia,SUN2013,zjw}]\label{lemma2.7}
Let $v\in L^s(0,T;L^r)$, $\dive v = 0$, $v \cdot n = 0$ and $\rho\in C([0,T];L^2)\cap L^2(0,T;H^1)$ be the weak solution of equation $\eqref{equation1.8}_1$, $\alpha\leq \rho\leq\beta$. Let $\rho$ satisfy $n\cdot \nabla\rho=0$ on $\partial\Omega$ provided $\Omega\subset\RR^3$ is a bounded domain with smooth boudary. Suppose that $\rho_0\in C^{\gamma_0}(\overline\Omega)$ for some $\gamma_0\in (0,1)$, then $\rho$ is H$\mathit{\ddot{o}}$lder continuous. More precisely, $\rho\in C^{\gamma,\frac{\gamma}{2}}(\overline Q_T)$, for some $\gamma$ depending only on $\gamma_0$, $\alpha$ and $\beta$.
\end{Lemma}

\begin{Lemma}\label{lemma2.8}
Let $\Omega$ be a bounded simply connnected domain in $\mathbb{R}^3$ with smooth boundary. Let $(u,p)$ be a strong solution of the following Stokes type system,
\begin{equation}\label{equation28}
\begin{cases}
-\mu(x)\Delta u+\nabla p=F, &x\in\Omega,\\
\dive u=0, &x\in\Omega,
\end{cases}
\end{equation}
where $p$ is normalized by the condition $\int p=0$, $F\in L^2$ and 
$$0<\underline\mu \leq \mu(x)\leq \overline\mu<\infty,\quad\mu(x)\in C(\overline\Omega).$$
Then, we have the following results:
\begin{enumerate}
\item[(1)] 
If $u$ satisfies the boundary condition $u\cdot n=0,\,\curle u\times n=\Phi$ on $\partial\Omega$, where $\Phi\in H^1$ is a function defined on $\Omega$. Then there exists a positive constant $C$ depending only on $\Omega$ such that
\begin{equation}\label{2.15}
\norm{u}_{H^2} +\norm{p}_{H^1}\leq C(\normf{F}_{L^2}+\norm{\Phi}_{H^1}).
\end{equation}
\item[(2)] 
 If $u$ satisfies the boundary condition $u=\Phi$ on $\partial\Omega$, where $\Phi\in H^2$ is a function defined on $\Omega$. Then there exists a positive constant $C$ depending only on $\Omega$ such that
\begin{equation}\label{2.16}
\norm{u}_{H^2} +\norm{p}_{H^1}\leq C(\normf{F}_{L^2}+\norm{\Phi}_{H^2}).
\end{equation}
\end{enumerate}
\end{Lemma}

\begin{proof}
The proof of Lemma \ref{lemma2.8} is an easy consequence of the freezing point argument, since we already have the conclusion when $\mu\equiv \text{constant}$ from the Lemma \ref{lemma2.3}.
\end{proof}

At last, in subsection \ref{P12}, we need the following lemma.
\begin{Lemma}[Simon \cite{novotny, simon}]\label{lemma221}
Let $X\hookrightarrow B\hookrightarrow Y$ be three Banach spaces with compact imbedding $X\hookrightarrow\hookrightarrow Y$. Further, let there eixst $0<\theta<1$ and $M>0$ such that
\begin{equation*}
\norm{v}_{B}\leq M\norm{v}_X^{1-\theta}\norm{v}_Y^\theta,\,\,\, \text{for all}\,\,v\in X\cap Y.
\end{equation*}
Denote for $T>0$,
\begin{equation*}
W(0,T):= W^{s_0,r_0}(0,T;X)\cap W^{s_1,r_1}(0,T;Y)
\end{equation*}
with $s_0,s_1\in\RR$, $r_1,r_0\in [1,\infty]$, and 
\begin{equation*}
s_\theta:=(1-\theta)s_0+\theta s_1,\,\,\frac{1}{r_\theta}:=\frac{1-\theta}{r_0}+\frac{\theta}{r_1},\,\,s^*:=s_\theta-\frac{1}{r_\theta}.
\end{equation*}
Assume that $s_\theta>0$ and $F$ is a bounded set in $W(0,T)$.
\begin{enumerate}
\item[(1)]
If $s^*\leq 0$, then $F$ is precompact in $L^p(0,T;B)$ for all $1\leq p<-\frac{1}{s^*}$.
\item[(2)]
If $s^*> 0$, then $F$ is precompact in $C([0,T];B)$. 
\end{enumerate}
\end{Lemma}

\section{Proof of Theorem \ref{Theorem1.1}}\label{section3}
In this section, we assume $u_0\in  C^\infty(\overline\Omega)\cap H^1_0$ (or $H^1_\omega$). We always suppose that the assumptions in Theorem \ref{Theorem1.1} hold. In the following proof, in order to simplify the notation, we denote by $\varepsilon_i$, $i\in\NN_+$, the arbitrarily small number belongs to $(0, 1/2]$ and we use the subscript $C_{\varepsilon_i}$ to emphasize the dependency of the constant $C$ on $\varepsilon_i$.

\subsection{A Priori Estimates}

\subsubsection{Case \eqref{equation1.6}}

The key of the proof is deriving the following proposition. Using
the idea from \cite{huang,zhang2015}, we first assume the bounds \eqref{condition2.5} and obtain the a priori estimates of $(\rho,v)$ (see below). Then, using the a priori estimates in Lemma \ref{Lemma2.2}--\ref{lemma33} leads to a smaller bounds \eqref{32}, which means that we can close the energy estimates. 

\begin{Proposition}\label{prop3.1}
There exists a positive constant $\delta$ depending on $\Omega$, $c_0$, $\alpha$ and $\beta$ such that, if $\norm{\nabla u_0}_{L^2}\leq \delta$ and
\begin{equation}\label{condition2.5}
\sup_{t\in[0,T]}\norm{\nabla\rho}_{L^6}\leq 2,\quad\int_0^T\left(\norm{\nabla v}_{L^2}^4+\norm{\Delta\rho}_{L^2}^4\right)\,dt\leq 2\norm{\nabla u_0}^2_{L^2},
\end{equation}
then, one has
\begin{equation}\label{32}
\sup_{t\in[0,T]}\norm{\nabla\rho}_{L^6}\leq 1,\quad\int_0^T\left(\norm{\nabla v}_{L^2}^4+\norm{\Delta\rho}_{L^2}^4\right)\,dt\leq \norm{\nabla u_0}^2_{L^2}.
\end{equation}
\end{Proposition}

We first come to prove the lower order estimates of $(\rho,v)$.
\begin{Lemma}\label{Lemma2.2}
Let $(\rho,v)$ be a smooth solution of \eqref{equation1.18}, then $\alpha\leq\rho\leq\beta$ and there exist some positive constant $C$ depending only on $\Omega$, $c_0$, $\alpha$ and $\beta$ such that, for all $T\in (0,\infty)$,
\begin{equation}\label{2.7}
\sup_{t\in[0,T]}\norm{\rho-(\rho_0)_\Omega}_{L^2}^2+\int_0^T\norm{\nabla\rho}_{L^2}^2\,dt\leq C\norm{\nabla u_0}_{L^2}^2.
\end{equation}
Furthermore, if $\norm{\nabla u_0}_{L^2}\leq 1$ and the condition \eqref{condition2.5} holds, one has
\begin{gather}
\sup_{t\in[0,T]}\norm{\nabla\rho}_{L^2}^2+\int_0^T\left(\norm{\nabla\rho}_{L^3}^4+ \norm{\Delta\rho}_{L^2}^2\right)\,dt\leq C\norm{\nabla u_0}^2_{L^2},\label{2.8}\\
\sup_{t\in[0,T]}\norm{v}^2_{L^2}+\int_0^T\left(\norm{v}_{L^3}^4+\norm{\nabla v}_{L^2}^2\right)\,dt\leq C\norm{\nabla u_0}^2_{L^2}.\label{2.9}
\end{gather}
\end{Lemma}

\begin{proof}
First of all, $\alpha\leq \rho\leq \beta$ is a consequence of the standard maximal principle. Next, multiplying $\rho-(\rho_0)_\Omega$ on both sides of $\eqref{equation1.18}_1$ and integrating over $\Omega$, one has
\begin{equation*}
\left(\norm{\rho-(\rho_0)_\Omega}_{L^2}^2\right)_t+\nu\norm{\nabla\rho}_{L^2}^2\leq 0.
\end{equation*}
Therefore, \eqref{2.7} is an easy consequence of Gr$\mathrm{\ddot{o}}$nwall's inequality and the control
$$\norm{\rho_0-(\rho_0)_\Omega}_{L^2}\leq C\norm{\nabla\rho_0}_{L^2}\leq C\norm{\nabla u_0}_{L^2}.$$

To prove \eqref{2.8}, multiplying $-\Delta\rho$ and integrating over $\Omega$, one has
\begin{equation}\label{eq3.7}
\begin{aligned}
\left(\int \frac{1}{2}|\nabla\rho|^2\right)_t+\int c_0\rho^{-1} |\Delta\rho|^2&=\int (v\cdot\nabla\rho)\Delta\rho+\int c_0\rho^{-2} |\nabla\rho|^2\Delta\rho\\
&:=\sum_{i=1}^2 G_i,
\end{aligned}
\end{equation}
where, using Lemma \ref{Lemma221},
\begin{equation}
\begin{cases}
|G_1|\leq C\norm{v}_{L^6}\norm{\nabla\rho}_{L^3}\norm{\Delta\rho}_{L^2}\leq C_{\varepsilon_1}\norm{\nabla v}^4_{L^2}\norm{\nabla\rho}_{L^2}^2+\varepsilon_1\norm{\Delta\rho}_{L^2}^2,\\
|G_2|\leq C\norm{\nabla\rho}_{L^6}\norm{\nabla\rho}_{L^3}\norm{\Delta\rho}_{L^2}\leq C_{\varepsilon_2}\norm{\Delta\rho}^4_{L^2}\norm{\nabla\rho}_{L^2}^2+\varepsilon_2\norm{\Delta\rho}_{L^2}^2.
\end{cases}
\end{equation}
Thus, we have
\begin{equation}\label{33.8}
\left(\norm{\nabla\rho}_{L^2}^2\right)_t+\nu\norm{\Delta\rho}_{L^2}^2\leq C\left(\norm{\Delta\rho}_{L^2}^4+\norm{\nabla v}^4_{L^2}\right)\norm{\nabla\rho}_{L^2}^2.
\end{equation}
Applying the Gr$\mathrm{\ddot{o}}$nwall's inequality on \eqref{33.8} and using cndition \eqref{condition2.5}, we obtain \eqref{2.8}.

With help of \eqref{2.7}--\eqref{2.8}, we next come to the proof of \eqref{2.9}. Multiplying $v$ on both sides of $\eqref{equation1.18}_2$ and integrating over $\Omega$, one has
\begin{equation}\label{218}
\begin{aligned}
\left(\frac{1}{2}\int \rho|v|^2\right)_t-\int \dive [2\mu(\rho)D(v)]\cdot v &=- \int 2c_0\mu(\rho)\nabla^2\rho^{-1}: \nabla v+ \int c_0 \rho v \cdot \nabla v\cdot \nabla\rho^{-1}\\
&\quad+\int c_0^2 \rho \nabla\rho^{-1} \cdot \nabla v\cdot \nabla\rho^{-1}\\
&\quad+\int_\partial 2c_0\mu(\rho)n\cdot \nabla^2\rho^{-1}\cdot v\\
&:=\sum_{i=1}^4H_i.
\end{aligned}
\end{equation}
For the second term on the left-hand side, using $\Delta v=-\curle\curle v$ and Lemma \ref{Lemma221} and \ref{lemma23}, we have
\begin{equation}\label{219}
\begin{aligned}
-\int \dive [2\mu(\rho)D(v)]\cdot v &=\int \mu(\rho)\curle\curle v\cdot v- \int 2\mu'(\rho)\nabla\rho\cdot D(v)\cdot v\\
&=\int \mu(\rho)|\curle v|^2+\int_\partial \mu(\rho)v\cdot B\cdot v+\int_\partial c_0\mu(\rho)v\cdot B\cdot \nabla\rho^{-1} \\
&\quad+\int \curle v\cdot \left(\nabla\mu(\rho)\times v\right)- \int 2\nabla\mu(\rho)\cdot D(v)\cdot v\\
&\geq \underline\mu\left(\norm{\curle v}_{L^2}^2+\int_\partial v\cdot B\cdot v\right)-C\norm{v}_{H^1}\norm{\nabla\rho}_{H^1}\\
&\quad-C\norm{\nabla\rho}_{L^6}\norm{v}_{L^3}\norm{\nabla v}_{L^2}\\
&\geq \nu\norm{\nabla v}_{L^2}^2-C\left(\norm{\Delta\rho}_{L^2}^2+\norm{\Delta\rho}_{L^2}^4\norm{v}_{L^2}^2\right).
\end{aligned}
\end{equation}
Next, for $H_1$--$H_4$, one has, applying Lemma \ref{Lemma221},
\begin{equation}\label{220}
\begin{cases}
|H_1|&\!\!\!\!\leq C\left(\norm{\Delta\rho}_{L^2}+\norm{\nabla\rho}_{L^6}\norm{\nabla\rho}_{L^3}\right)\norm{\nabla v}_{L^2}\\
&\!\!\!\!\leq C_{\varepsilon_1}\left(\norm{\Delta\rho}_{L^2}^2+\norm{\Delta\rho}_{L^2}^4\norm{\nabla\rho}_{L^2}^2\right)+\varepsilon_1\norm{\nabla v}_{L^2}^2,\\
|H_2|&\!\!\!\!\leq C\norm{\nabla\rho}_{L^6}\norm{v}_{L^3}\norm{\nabla v}_{L^2}\leq C_{\varepsilon_2}\norm{\Delta\rho}_{L^2}^4\norm{v}_{L^2}^2+\varepsilon_2\norm{\nabla v}_{L^2}^2,\\
|H_3|&\!\!\!\!\leq C\norm{\nabla\rho}_{L^6}\norm{\nabla\rho}_{L^3}\norm{\nabla v}_{L^2}\leq C_{\varepsilon_3}\norm{\Delta\rho}_{L^2}^4\norm{\nabla\rho}_{L^2}^2+C_{\varepsilon_3}\norm{\Delta\rho}_{L^2}^2+\varepsilon_3\norm{\nabla v}_{L^2}^2.
\end{cases}
\end{equation}
and, using the fact that 
\begin{equation*}
v\cdot \nabla^2\rho^{-1}\cdot n=-v\cdot \nabla n\cdot \nabla \rho^{-1}\,\,\mathrm{on}\,\,\partial\Omega\times(0,T),
\end{equation*}
\begin{equation}\label{221}
\begin{aligned}
|H_4|&=\abs{\int_\partial 2c_0\mu(\rho)n\cdot \nabla^2\rho^{-1}\cdot v}=\abs{\int_\partial 2c_0\mu(\rho) v\cdot\nabla n\cdot \nabla\rho^{-1}}\\
&=\abs{\int_\partial 2c_0\mu(\rho)(v^\perp\times n)\cdot\nabla n\cdot \nabla\rho^{-1}}\\
&=\abs{\int 2c_0\mu(\rho)\curle v^\perp\cdot \nabla n\cdot \nabla\rho^{-1} -\int 2c_0v^\perp\cdot \curle\left[\mu(\rho)\nabla n\cdot \nabla\rho^{-1}\right]}\\
&=\left|\int 2c_0\mu(\rho)\curle v^\perp\cdot \nabla n\cdot \nabla\rho^{-1}+\int 2c_0\mu(\rho)v^\perp\cdot \curle\left(\nabla n\cdot \nabla\rho^{-1}\right)\right.\\
&\quad-\left.\int 2c_0\nabla\mu(\rho)\times\left(\nabla n\cdot \nabla\rho^{-1}\right)\cdot v^\perp\right|\\
&\leq C\norm{v}_{H^1}\norm{\nabla\rho}_{H^1}+C\norm{\nabla\rho}_{L^3}^2\norm{v}_{L^3}\leq C_{\varepsilon_4}\left(\norm{\Delta\rho}_{L^2}^2+\norm{\nabla\rho}_{L^3}^4\right)+\varepsilon_4\norm{\nabla v}_{L^2}^2.
\end{aligned}
\end{equation}

Combining \eqref{219}--\eqref{221}, we deduce from \eqref{218} that
\begin{equation}\label{3314}
\begin{aligned}
&\left(\norm{\sqrt\rho v}^2_{L^2}\right)_t+\nu\norm{\nabla v}_{L^2}^2\\
&\leq C\left(\norm{\Delta\rho}_{L^2}^4\norm{\sqrt\rho v}_{L^2}^2+\norm{\Delta\rho}_{L^2}^2+\norm{\Delta\rho}_{L^2}^4\norm{\nabla\rho}_{L^2}^2+\norm{\nabla\rho}_{L^3}^4\right).
\end{aligned}
\end{equation}
Using Gr$\mathrm{\ddot{o}}$nwall's inequality, Lemma \ref{Lemma221}, \eqref{1.17}, \eqref{condition2.5} and \eqref{2.8},
\begin{equation*}
\begin{aligned}
\sup_{t\in[0,T]}\norm{v}^2_{L^2}+\int_0^T\left(\norm{v}_{L^3}^4+\norm{\nabla v}_{L^2}^2\right)\,dt&\leq C\left(\norm{\nabla u_0}_{L^2}^2+\norm{v_0}_{L^2}^2\right)\\
&\leq C\norm{\nabla u_0}_{L^2}^2,
\end{aligned}
\end{equation*}
which gives \eqref{2.9}. Thus, we complete the proof of Lemma \ref{Lemma2.2}.
\end{proof}

Next, we prove the higher order estimates for $(\rho, v)$, that is,
\begin{Lemma}\label{lemma33}
Let $(\rho,v,\pi_1)$ be a smooth solution of \eqref{equation1.18}. Suppose that $\norm{\nabla u_0}_{L^2}\leq 1$ and the condition \eqref{condition2.5} holds, then there exist some positive constants $C$ depending only on $\Omega$, $c_0$, $\alpha$ and $\beta$ such that, for all $T\in (0,\infty)$,
\begin{gather}
\sup_{t\in[0,T]}\cP(t)+\int_0^T\left(\cQ(t)+\norm{\pi}_{H^1}^2\right)\,dt\leq C\norm{\nabla u_0}^2_{L^2}.\label{2.12}
\end{gather}
where
\begin{gather*}
\cP(t):=\norm{\nabla v}_{L^2}^2+\norm{\Delta\rho}^2_{L^2}+\norm{\rho_t}_{L^2}^2,\\
\cQ(t):=\norm{v_t}_{L^2}^2+\norm{v}_{H^2}^2+\norm{\nabla\Delta\rho}_{L^2}^2+\norm{\nabla\rho_t}_{L^2}^2.
\end{gather*}
\end{Lemma}

\begin{proof}
We first apply $-\nabla\Delta\rho\cdot \nabla$ on both sides of $\eqref{equation1.18}_1$ and, then, integrate over $\Omega$, we have
\begin{equation}\label{2.20}
\begin{aligned}
\left(\int \frac{1}{2}|\Delta\rho|^2\right)_t+\int c_0\rho^{-1}|\nabla\Delta\rho|^2&=\int \nabla\Delta\rho\cdot \nabla v\cdot \nabla\rho+\int v\cdot \nabla^2\rho\cdot\nabla\Delta\rho\\
&\quad+\int \nabla \left[\frac{c_0}{\rho^2}|\nabla\rho|^2\right]\cdot\nabla\Delta\rho+\int \frac{c_0}{\rho^2}\Delta\rho\nabla\rho\cdot\nabla\Delta\rho\\
&:=\sum_{i=1}^4 I_i,
\end{aligned}
\end{equation}
where, applying Lemma \ref{Lemma221},
\begin{equation}\label{2.21}
\begin{cases}
|I_1|&\!\!\!\!\leq C\norm{\nabla v}_{L^3}\norm{\nabla\rho}_{L^6}\norm{\nabla\Delta\rho}_{L^2}\leq C_{\varepsilon_1}\norm{\Delta\rho}_{L^2}^4\norm{\nabla v}_{L^2}^2+\varepsilon_1\left(\norm{\nabla\Delta\rho}_{L^2}^2+\norm{v}_{H^2}^2\right),\\
|I_2|&\!\!\!\!\leq  C\norm{v}_{L^6}\norm{\Delta\rho}_{L^3}\norm{\nabla\Delta\rho}_{L^2}\leq C_{\varepsilon_2}\norm{\nabla v}_{L^2}^4\norm{\Delta\rho}_{L^2}^2+\varepsilon_2\norm{\nabla\Delta\rho}_{L^2}^2,\\
|I_3|&\!\!\!\!\leq C\left(\norm{\nabla\rho}_{L^6}^3+\norm{\nabla\rho}_{L^6}\normf{\nabla^2\rho}_{L^3}\right)\norm{\nabla\Delta\rho}_{L^2}\\
&\!\!\!\!\leq C_{\varepsilon_3}\norm{\Delta\rho}_{L^2}^4\norm{\Delta\rho}_{L^2}^2+\varepsilon_3\norm{\nabla\Delta\rho}_{L^2}^2,\\
|I_4|&\!\!\!\!\leq C\norm{\nabla\rho}_{L^6}\normf{\Delta\rho}_{L^3}\norm{\nabla\Delta\rho}_{L^2}\leq C_{\varepsilon_4}\norm{\Delta\rho}_{L^2}^4\norm{\Delta\rho}_{L^2}^2+\varepsilon_4\norm{\nabla\Delta\rho}_{L^2}^2.
\end{cases}
\end{equation}
Thus, substituting \eqref{2.21} into \eqref{2.20} leads to
\begin{equation}\label{318}
\begin{aligned}
\left(\norm{\Delta\rho}_{L^2}^2\right)_t+\nu \norm{\nabla\Delta\rho}_{L^2}^2&\leq C_\varepsilon\left(\norm{\nabla v}_{L^2}^4+\norm{\Delta\rho}_{L^2}^4\right)\left(\norm{\Delta\rho}_{L^2}^2+\norm{\nabla v}_{L^2}^2\right)+\varepsilon\norm{v}_{H^2}^2
\end{aligned}
\end{equation}

For the higher order estimates of $v$, multiplying $v_t$ on both sides of $\eqref{equation1.18}_2$ and integrating over $\Omega$ lead to
\begin{equation}\label{2.23}
\begin{aligned}
\int \rho|v_t|^2-\int\dive[2\mu(\rho)D(v)]\cdot v_t&=-\int\rho u\cdot \nabla v\cdot v_t+ \int c_0\dive{\left[2\mu(\rho)\nabla^2\rho^{-1}\right]}\cdot v_t\\
&\quad-\int c_0 \dive{\left(\rho v\otimes\nabla\rho^{-1}\right)}\cdot v_t\\
&\quad -\int c_0^2 \dive{\left(\rho \nabla\rho^{-1}\otimes\nabla\rho^{-1}\right)}\cdot v_t\\
&:=\sum_{i=1}^4 J_i.
\end{aligned}
\end{equation}
For the second term on the left-hand side, we have
\begin{align}
-\int\dive[2\mu(\rho)D(v)]\cdot v_t&=\int \mu(\rho)\curle\curle v\cdot v_t-\int 2\nabla\mu(\rho)\cdot D(v)\cdot v_t\notag\\
&=\int_\partial \mu(\rho)v_t\cdot B\cdot v+\int_\partial c_0 \mu(\rho)v_t\cdot B\cdot \nabla\rho^{-1}+\int \mu(\rho)\curle v\cdot \curle v_t\notag\\
&\quad+\int \curle v\cdot\left[\nabla\mu(\rho)\times v_t\right]-\int 2\nabla\mu(\rho)\cdot D(v)\cdot v_t\notag\\
&=\frac{1}{2}\left(\int_\partial \mu(\rho)v\cdot B\cdot v+\int \mu(\rho)|\curle v|^2\right)_t+\int_\partial c_0 \mu(\rho)v_t\cdot B\cdot \nabla\rho^{-1} \notag\\
&\quad-\int_\partial \frac{1}{2}\mu(\rho)_tv\cdot B\cdot v-\int \frac{1}{2}\mu(\rho)_t |\curle v|^2\notag\\
&\quad+\int \curle v\cdot\left[\nabla\mu(\rho)\times v_t\right]-\int 2\nabla\mu(\rho)\cdot D(v)\cdot v_t\notag\\
&:= \frac{1}{2}\left(\int_\partial \mu(\rho)v\cdot B\cdot v+\int \mu(\rho)|\curle v|^2\right)_t +\sum_{i=1}^5K_i.\label{2.25}
\end{align}
However, For $K_1$--$K_5$, we have, using Lemma \ref{Lemma221},
\begin{equation}\label{equ3.21}
\begin{aligned}
K_1&=-\int_\partial c_0 \frac{\mu(\rho)}{\rho^2}(v_t^\perp \times n)\cdot B\cdot \nabla\rho\\
&=\int c_0 \frac{\mu(\rho)}{\rho^2}\curle v_t^\perp\cdot B\cdot \nabla\rho-\int c_0 v_t^\perp\cdot \left[\nabla\frac{\mu(\rho)}{\rho^2}\times (B\cdot \nabla\rho)\right]\\
&\quad-\int c_0\frac{\mu(\rho)}{\rho^2} v_t^\perp\cdot \curle(B\cdot \nabla\rho)\\
&=M'(t)-\int c_0\curle v^\perp\cdot B\cdot  \left[\frac{\mu(\rho)}{\rho^2}\nabla\rho\right]_t-\int c_0 v_t^\perp\cdot \left[\nabla\frac{\mu(\rho)}{\rho^2}\times (B\cdot \nabla\rho)\right]\\
&\quad-\int c_0\frac{\mu(\rho)}{\rho^2} v_t^\perp\cdot \curle(B\cdot \nabla\rho)\\
&\leq M'(t)+C\norm{\nabla v}_{L^2}\left(\norm{\nabla\rho_t}_{L^2}+\norm{\rho_t}_{L^6}\norm{\nabla\rho}_{L^3}\right)\\
&\quad+C\norm{v_t}_{L^2}\left(\norm{\nabla\rho}_{L^6}\norm{\nabla\rho}_{L^3}+\norm{\Delta\rho}_{L^2}\right)\\
&\leq M'(t)+C_{\varepsilon_1}\left(\norm{\nabla\rho}^4_{L^3}\norm{\nabla v}_{L^2}^2+\norm{\nabla v}^2_{L^2}\right)+\varepsilon_1\norm{\nabla\rho_t}_{L^2}^2\\
&\quad+C_{\varepsilon_2}\left(\norm{\Delta\rho}^4_{L^2}\norm{\nabla\rho}^2_{L^2}+\norm{\Delta\rho}^2_{L^2}\right)+\varepsilon_2\norm{v_t}^2_{L^2},
\end{aligned}
\end{equation}
where
\begin{equation*}
M(t):=\int c_0 \frac{\mu(\rho)}{\rho^2}\curle v^\perp\cdot B\cdot \nabla\rho,
\end{equation*}
and
\begin{equation}\label{equ3.22}
\begin{aligned}
K_2&=-\int_\partial \frac{1}{2}\mu(\rho)_tv\cdot B\cdot v=\int_\partial \frac{1}{2}\mu(\rho)_t (n\times v^\perp)\cdot B\cdot v\\
&=\int \frac{1}{2}\mu(\rho)_t\curle v^\perp \cdot B\cdot v-\int \frac{1}{2} v^\perp \cdot\left[ \nabla\mu(\rho)_t\times (B\cdot v)\right]-\int \frac{1}{2}\mu(\rho)_t v^\perp \cdot \curle(B\cdot v)\\
&\leq C\norm{\rho_t}_{L^6}\norm{v}_{L^3}\norm{\nabla v}_{L^2}+C\norm{v}_{L^6}\norm{v}_{L^3}\norm{\nabla\rho_t}_{L^2}+C\norm{v}_{L^3}\norm{v}_{L^6}\norm{\nabla\rho}_{L^3}\norm{\rho_t}_{L^6}\\
&\leq C_{\varepsilon_3}\left(\norm{v}_{L^3}^4\norm{\nabla v}_{L^2}^2+\norm{\nabla\rho}_{L^3}^4\norm{\nabla v}_{L^2}^2+\norm{\nabla v}_{L^2}^2\right)+\varepsilon_3\norm{\nabla\rho_t}_{L^2}^2
\end{aligned}
\end{equation}
and 
\begin{equation}\label{2.28}
\begin{cases}
|K_3|\leq C\norm{\rho_t}_{L^6}\norm{\nabla v}_{L^3}\norm{\nabla v}_{L^2}\leq C_{\varepsilon_4}\norm{\nabla v}^4_{L^2}\norm{\nabla v}_{L^2}^2+\varepsilon_4\left(\norm{v}_{H^2}^2+\norm{\nabla\rho_t}_{L^2}^2\right),\\
|K_4|\leq C\norm{\nabla v}_{L^3}\norm{\nabla\rho}_{L^6}\norm{v_t}_{L^2}\leq C_{\varepsilon_5}\norm{\Delta\rho}^4_{L^2}\norm{\nabla v}_{L^2}^2+\varepsilon_5\left(\norm{v}_{H^2}^2+\norm{v_t}_{L^2}^2\right),\\
|K_5|\leq C\norm{\nabla v}_{L^3}\norm{\nabla\rho}_{L^6}\norm{v_t}_{L^2}\leq C_{\varepsilon_6}\norm{\Delta\rho}^4_{L^2}\norm{\nabla v}_{L^2}^2+\varepsilon_6\left(\norm{v}_{H^2}^2+\norm{v_t}_{L^2}^2\right).
\end{cases}
\end{equation}
Combining \eqref{2.25}--\eqref{2.28}, we have
\begin{equation}\label{228}
\begin{aligned}
-\int\dive[2\mu(\rho)D(v)]\cdot v_t &\geq \frac{1}{2}\left(\int_\partial \mu(\rho)v\cdot B\cdot v+\int \mu(\rho)|\curle v|^2\right)_t+M'(t)\\
&\quad-C_\varepsilon\left(\norm{\nabla\rho}_{L^3}^4+\norm{\Delta\rho}_{L^2}^4+\norm{v}_{L^3}^4+\norm{\nabla v}_{L^2}^4\right)\norm{\nabla v}_{L^2}^2\\
&\quad-C_\varepsilon\norm{\nabla v}^2_{L^2}-C_{\varepsilon}\left(\norm{\Delta\rho}^4_{L^2}\norm{\nabla\rho}^2_{L^2}+\norm{\Delta\rho}^2_{L^2}\right)\\
&\quad-\varepsilon\left(\norm{v}_{H^2}^2+\norm{v_t}_{L^2}^2\right).
\end{aligned}
\end{equation}

Next, we turn to estimate $J_1$--$J_4$ and apply Lemma \ref{Lemma221}, that is, 
\begin{equation}\label{2.29}
\begin{cases}
|J_1|\leq C\left(\norm{\nabla\rho}_{L^6}+\norm{v}_{L^6}\right)\norm{\nabla v}_{L^3}\norm{v_t}_{L^2}\\
\quad\quad\!\leq C_{\varepsilon_1}\left(\norm{\Delta\rho}_{L^2}^4+\norm{\nabla v}_{L^2}^4\right)\norm{\nabla v}_{L^2}^2+\varepsilon_1\left(\norm{v}_{H^2}^2+\norm{v_t}_{L^2}^2\right),\\
|J_2|\leq C\left(\norm{\nabla\rho}_{L^6}^3+\norm{\nabla\rho}_{L^6}\norm{\Delta\rho}_{L^3}+\norm{\nabla\Delta\rho}_{L^2}\right)\norm{v_t}_{L^2}\\
\quad\quad\!\leq C_{\varepsilon_2}\left(\norm{\Delta\rho}_{L^2}^4\norm{\Delta\rho}_{L^2}^2+\norm{\nabla\Delta\rho}_{L^2}^2\right)+\varepsilon_2\norm{v_t}_{L^2}^2,\\
|J_3|\leq C\left(\norm{\nabla\rho}_{L^6}^2\norm{v}_{L^6}+\norm{\nabla\rho}_{L^6}\norm{\nabla v}_{L^3}+\norm{\Delta\rho}_{L^3}\norm{v}_{L^6}\right)\norm{v_t}_{L^2}\\
\quad\quad\!\leq C_{\varepsilon_3}\left(\norm{\Delta\rho}_{L^2}^4\norm{\nabla v}_{L^2}^2+\norm{\nabla v}_{L^2}^4\norm{\Delta\rho}_{L^2}^2\right)\\
\quad\quad\quad\!+\varepsilon_3\left(\norm{\Delta v}_{L^2}^2+\norm{v_t}_{L^2}^2+\norm{\nabla\Delta\rho}_{L^2}^2\right),\\
|J_4|\leq C\left(\norm{\nabla\rho}_{L^6}^3+\norm{\nabla\rho}_{L^6}\norm{\Delta\rho}_{L^3}\right)\norm{v_t}_{L^2}\\
\quad\quad\!\leq C_{\varepsilon_4}\left(\norm{\Delta\rho}_{L^2}^4\norm{\Delta\rho}_{L^2}^2+\norm{\nabla\Delta\rho}_{L^2}^2\right)+\varepsilon_4\norm{v_t}_{L^2}^2.\\
\end{cases}
\end{equation}
Now, substituting \eqref{228} and \eqref{2.29} into \eqref{2.23}, one  can deduce that
\begin{equation}\label{2.30}
\begin{aligned}
&\left(\int_\partial \mu(\rho)v\cdot B\cdot v+\int \mu(\rho)|\curle v|^2\right)_t +\nu\norm{v_t}_{L^2}^2+M'(t)\\
&\leq C_\varepsilon\left(\norm{\nabla\rho}_{L^3}^4+\norm{\Delta\rho}_{L^2}^4+\norm{v}_{L^3}^4+\norm{\nabla v}_{L^2}^4+1\right)\norm{\nabla v}_{L^2}^2\\
&\quad+C_{\varepsilon}\left(\norm{\Delta\rho}^4_{L^2}\norm{\Delta\rho}^2_{L^2}+\norm{\nabla\Delta\rho}^2_{L^2}\right)+\varepsilon\left(\norm{v}_{H^2}^2+\norm{\nabla\rho_t}_{L^2}^2\right).
\end{aligned}
\end{equation}
For simplicity, we rewrite \eqref{2.30} as 
\begin{equation}\label{327}
\begin{aligned}
\left(\norm{\nabla v}_{L^2}^2\right)_t +\nu\norm{v_t}_{L^2}^2+M'(t)&\leq C_\varepsilon\left(\norm{\nabla\rho}_{L^3}^4+\norm{\Delta\rho}_{L^2}^4+\norm{v}_{L^3}^4+\norm{\nabla v}_{L^2}^4+1\right)\norm{\nabla v}_{L^2}^2\\
&\quad+C_{\varepsilon}\left(\norm{\Delta\rho}^4_{L^2}\norm{\Delta\rho}^2_{L^2}+\norm{\nabla\Delta\rho}^2_{L^2}\right)+\varepsilon\norm{v}_{H^2}^2,\\
&\leq C_\varepsilon\left(\norm{\Delta\rho}_{L^2}^4+\norm{\nabla v}_{L^2}^4\right)\left(\norm{\nabla v}_{L^2}^2+\norm{\Delta\rho}^2_{L^2}\right)\\
&\quad+C_{\varepsilon}\left(\norm{\nabla v}^2_{L^2}+\norm{\nabla\Delta\rho}^2_{L^2}\right)+\varepsilon\left(\norm{v}_{H^2}^2+\norm{\nabla\rho_t}_{L^2}^2\right),\\
\end{aligned}
\end{equation}
since, from the positivity of $B$ and Lemma \ref{lemma22}
$$\int_\partial \mu(\rho)v\cdot B\cdot v\geq 0,\,\,\,\,\int \mu(\rho)|\curle v|^2\sim \norm{\nabla v}_{L^2}^2,$$
and they do not influent the results after applying the Gr$\mathrm{\ddot{o}}$nwall's inequality for \eqref{2.30}.

We still need to estimate $\norm{v}_{H^2}$. To get this, we convert $\eqref{equation1.18}_2$ into the form
\begin{equation}\label{328}
-\dive[2\mu(\rho)D(v)]+\nabla \pi=F,
\end{equation}
where
\begin{equation}\label{F}
\begin{aligned}
F&:=-\rho v_t -\rho u\cdot\nabla v +c_0\dive{\left[2\mu(\rho)\nabla^2\rho^{-1}\right]}- c_0 \dive{\left(\rho v\otimes\nabla\rho^{-1}\right)}\\
&\,\,\quad -c_0^2 \dive{\left(\rho \nabla\rho^{-1}\otimes\nabla\rho^{-1}\right)}+c_0\nabla(\log\rho)_t.
\end{aligned}
\end{equation}
In order to use Lemma \ref{lemma26}, from the embedding $L^2\hookrightarrow H^{-1}$, one should estimate $\norm{F}_{L^2}$, that is,
\begin{equation}\label{329}
\begin{aligned}
\norm{F}^2_{L^2}&\leq C\left(\norm{v_t}^2_{L^2}+\norm{\nabla\rho_t}^2_{L^2}\right)+C\left[\left(\norm{\Delta\rho}_{L^2}^4+\norm{\nabla v}_{L^2}^4\right)\norm{\Delta\rho}_{L^2}^2+\norm{\nabla\Delta\rho}_{L^2}^2\right]\\
&\quad+C_{\varepsilon_1}\norm{\Delta\rho}_{L^2}^4\norm{\nabla v}_{L^2}^2+\varepsilon_1\norm{v}_{H^2}^2,
\end{aligned}
\end{equation}
where we have used 
$$\norm{\nabla(\log\rho)_t}_{L^2}\leq C\left(\norm{\nabla\rho_t}_{L^2}+\norm{\rho_t}_{L^3}\norm{\nabla\rho}_{L^6}\right)\leq C\norm{\nabla\rho_t}_{L^2}$$
from Lemma \ref{Lemma221} and condition \eqref{condition2.5}.

On the other hand, in this case, $\Phi:=-B\cdot(v+c_0\nabla\rho^{-1})$, where $\Phi$ as in Lemma \ref{lemma26}. Hence, applying Poincar\'e's inequality leads to
\begin{equation}\label{330}
\begin{aligned}
\norm{\Phi}^2_{H^1}&\leq C\left(\norm{v}_{H^1}^2+\norm{\nabla\rho}_{H^1}^2+\norm{\nabla\rho}_{L^3}^2\norm{\nabla\rho}_{L^6}^2\right)\\
&\leq C\left(\norm{\nabla v}_{L^2}^2+\norm{\Delta\rho}_{L^2}^2+\norm{\Delta\rho}_{L^2}^4\norm{\nabla\rho}_{L^2}^2\right).
\end{aligned}
\end{equation}
Combining \eqref{329}--\eqref{330} and using Lemma \ref{lemma26}, condition \eqref{condition2.5} and Poincar\'e's inequality, we deduce from \eqref{328} that
\begin{equation*}
\begin{aligned}
\norm{v}_{H^2}^2+\norm{\pi}^2_{H^1}&\leq C\left[\norm{\nabla\rho}^4_{L^6}\norm{\nabla v}^2_{L^2}+\left(1+\norm{\nabla\rho}^4_{L^6}\right)\left(\norm{F}^2_{L^2}+\norm{\Phi}^2_{H^1}\right)\right]\\
&\leq C\left(\normf{F}^2_{L^2}+\norm{\Phi}^2_{H^1}+\norm{\nabla v}^2_{L^2}\right)\\
&\leq C\left(\norm{v_t}^2_{L^2}+\norm{\nabla\rho_t}^2_{L^2}\right)+C\left(\norm{\Delta\rho}_{L^2}^4+\norm{\nabla v}_{L^2}^4\right)\norm{\Delta\rho}_{L^2}^2+C\norm{\nabla\Delta\rho}_{L^2}^2\\
&\quad+C\norm{\nabla v}_{L^2}^2+C_{\varepsilon}\norm{\Delta\rho}_{L^2}^4\norm{\nabla v}_{L^2}^2+\varepsilon\norm{v}_{H^2}^2,
\end{aligned}
\end{equation*}
which gives 
\begin{equation}\label{331}
\begin{aligned}
\norm{v}_{H^2}^2+\norm{\pi}^2_{H^1}&\leq C\left(\norm{v_t}^2_{L^2}+\norm{\nabla\rho_t}^2_{L^2}\right)+C\left(\norm{\Delta\rho}_{L^2}^4+\norm{\nabla v}_{L^2}^4\right)\left(\norm{\Delta\rho}_{L^2}^2+\norm{\nabla v}_{L^2}^2\right)\\
&\quad+C\left(\norm{\nabla\Delta\rho}_{L^2}^2+\norm{\nabla v}_{L^2}^2\right)
\end{aligned}
\end{equation}
Combining \eqref{327} and \eqref{331}, one has
\begin{equation}\label{332}
\begin{aligned}
\left(\norm{\nabla v}_{L^2}^2\right)_t +\nu\norm{v_t}_{L^2}+\nu\norm{v}_{H^2}^2+M'(t)&\leq C_\varepsilon\left(\norm{\Delta\rho}_{L^2}^4+\norm{\nabla v}_{L^2}^4\right)\left(\norm{\nabla v}_{L^2}^2+\norm{\Delta\rho}^2_{L^2}\right)\\
&\quad+C_\varepsilon\left(\norm{\nabla\Delta\rho}^2_{L^2}+\norm{\nabla v}^2_{L^2}\right)+\varepsilon\norm{\nabla\rho_t}^2_{L^2}.
\end{aligned}
\end{equation}

At last, we come to estimate $\nabla\rho_t$. Applying $\rho_t\partial_t$ on both sides of $\eqref{equation1.18}_1$ and integrating over $\Omega$ yield that
\begin{align}
\left(\int\frac{1}{2}\abs{\rho_t}^2\right)_t +\int c_0\rho^{-1}\abs{\nabla\rho_t}^2 &=-\int (v_t\cdot\nabla\rho)\rho_t+ \int 2c_0 \rho^{-3}\abs{\rho_t}^2\abs{\nabla\rho}^2\notag\\
&\quad-\int  2c_0\rho^{-1}\left(\nabla\rho\cdot\nabla\rho_t\right)\rho_t-\int c_0\rho^{-1}\abs{\rho_t}^2\Delta\rho\notag\\
&:= \sum_{i=1}^4L_i.\label{33.5}
\end{align}
It follow from Lemma \ref{Lemma221} that
\begin{equation}\label{33.6}
\begin{cases}
|L_1|\leq \norm{v_t}_{L^2}\norm{\nabla\rho}_{L^6}\norm{\rho_t}_{L^3}\leq C_{\varepsilon_1}\norm{\Delta\rho}_{L^2}^4\norm{\rho_t}_{L^2}^2+\varepsilon_1\left(\norm{v_t}_{L^2}^2+\norm{\nabla\rho_t}_{L^2}^2\right),\\
|L_2|\leq \norm{\nabla\rho}_{L^2}\norm{\nabla\rho}_{L^6}\norm{\rho_t}_{L^3}\leq C_{\varepsilon_2}\norm{\Delta\rho}_{L^2}^4\norm{\rho_t}_{L^2}^2+\varepsilon_2\left(\norm{\nabla\rho}_{L^2}^2+\norm{\nabla\rho_t}_{L^2}^2\right),\\
|L_3|\leq \norm{\nabla\rho_t}_{L^2}\norm{\nabla\rho}_{L^6}\norm{\rho_t}_{L^3}\leq C_{\varepsilon_3}\norm{\Delta\rho}_{L^2}^4\norm{\rho_t}_{L^2}^2+\varepsilon_3\norm{\nabla\rho_t}_{L^2}^2,\\
|L_4|\leq \norm{\Delta\rho}_{L^2}\norm{\rho_t}_{L^6}\norm{\rho_t}_{L^3}\leq C_{\varepsilon_4}\norm{\Delta\rho}_{L^2}^4\norm{\rho_t}_{L^2}^2+\varepsilon_4\norm{\nabla\rho_t}_{L^2}^2.
\end{cases}
\end{equation}
Combining \eqref{33.5} and \eqref{33.6} leads to
\begin{equation}
\left(\norm{\rho_t}_{L^2}^2\right)_t+\nu\norm{\nabla\rho_t}_{L^2}^2\leq C_\varepsilon\norm{\Delta\rho}_{L^2}^4\norm{\rho_t}_{L^2}^2+\varepsilon\left(\norm{v_t}_{L^2}^2+\norm{\nabla\rho}_{L^2}^2\right).
\end{equation}
This, alonging with \eqref{332}, yields that
\begin{equation}\label{3336}
\begin{aligned}
&\left(\norm{\nabla v}_{L^2}^2+\norm{\rho_t}_{L^2}^2\right)_t +\nu\norm{v_t}_{L^2}^2+\nu\norm{v}_{H^2}^2+\nu\norm{\nabla\rho_t}_{L^2}^2+M'(t)\\
&\leq C\left(\norm{\Delta\rho}_{L^2}^4+\norm{\nabla v}_{L^2}^4\right)\cP(t)+C\left(\norm{\nabla\Delta\rho}^2_{L^2}+\norm{\nabla v}^2_{L^2}\right).
\end{aligned}
\end{equation}

On the other hand, combining \eqref{318} and \eqref{331} leads to
\begin{equation}
\begin{aligned}
\left(\norm{\Delta\rho}_{L^2}^2\right)_t+\nu \norm{\nabla\Delta\rho}_{L^2}^2&\leq C_\varepsilon\left(\norm{\nabla v}_{L^2}^4+\norm{\Delta\rho}_{L^2}^4\right)\left(\norm{\Delta\rho}_{L^2}^2+\norm{\nabla v}_{L^2}^2\right)\\
&\quad+\varepsilon\left(\norm{\nabla v}_{L^2}^2+\norm{v_t}_{L^2}^2+\norm{\nabla \rho_t}_{L^2}^2\right),
\end{aligned}
\end{equation}
that is,
\begin{equation}\label{333}
\begin{aligned}
\frac{\nu}{2} \norm{\nabla\Delta\rho}_{L^2}^2&\leq -\frac{\nu}{2} \norm{\nabla\Delta\rho}_{L^2}^2-\left(\norm{\Delta\rho}_{L^2}^2\right)_t\\
&\quad+ C_\varepsilon\left(\norm{\nabla v}_{L^2}^4+\norm{\Delta\rho}_{L^2}^4\right)\left(\norm{\Delta\rho}_{L^2}^2+\norm{\nabla v}_{L^2}^2\right)\\
&\quad+\varepsilon\left(\norm{\nabla v}_{L^2}^2+\norm{v_t}_{L^2}^2+\norm{\nabla \rho_t}_{L^2}^2\right).
\end{aligned}
\end{equation}
Thus, substituting \eqref{333} into \eqref{3336} and choosing $\varepsilon$ small enough, we obtain
\begin{equation*}
\begin{aligned}
&\left(\frac{2C}{\nu}\norm{\Delta\rho}_{L^2}^2+\norm{\nabla v}_{L^2}^2+\norm{\rho_t}_{L^2}^2\right)_t +\frac{\nu}{2}\left(\frac{2C}{\nu}\norm{\nabla\Delta\rho}_{L^2}^2+\norm{v_t}_{L^2}^2+\norm{\nabla\rho_t}_{L^2}^2\right)+M'(t)\\
&\leq C\left(\norm{\nabla v}_{L^2}^4+\norm{\Delta\rho}_{L^2}^4\right)\cP(t)+C\norm{\nabla v}_{L^2}^2,
\end{aligned}
\end{equation*}
or, equivalent to say, using the definition of $\cP(t),\,\cQ(t)$,
\begin{equation*}
\begin{aligned}
\cP'(t)+\nu\cQ(t)+M'(t)\leq C\left(\norm{\nabla v}_{L^2}^4+\norm{\Delta\rho}_{L^2}^4\right)\cP(t)+C\norm{\nabla v}_{L^2}^2,
\end{aligned}
\end{equation*}

Then, alonging with Lemma \ref{Lemma2.2} and using Gr$\mathrm{\ddot{o}}$nwall's inequality gives
\begin{equation}\label{334}
\sup_{t\in[0,T]}\cP(t)+\int_0^T\cQ(t),dt\leq C\norm{\nabla u_0}_{L^2}^2,
\end{equation}
where we have used the control
\begin{equation*}
\begin{aligned}
\norm{\rho_{t,0}}_{L^2}^2&\leq C\left(\norm{\nabla\rho_0}_{L^3}^2\norm{v_0}_{L^6}^2+\norm{\nabla\rho_0}_{L^3}^2\norm{\nabla\rho_0}_{L^6}^2+\norm{\Delta\rho_0}_{L^2}^2\right)\\
&\leq C\left(\norm{\Delta\rho_0}_{L^2}^2\norm{\nabla v_0}_{L^2}^2+\norm{\Delta\rho_0}_{L^2}^4+\norm{\Delta\rho_0}_{L^2}^2\right)\\
&\leq C\left(\norm{\nabla u_0}_{L^2}^4+\norm{\nabla u_0}_{L^2}^2\right)\leq C\norm{\nabla u_0}_{L^2}^2,
\end{aligned}
\end{equation*}
and the follwoing estimates
\begin{equation}\label{341}
\begin{aligned}
\sup_{t\in[0,T]}M(t)&\leq \varepsilon \sup_{t\in[0,T]}\norm{\nabla v}_{L^2}^2+C_\varepsilon\sup_{t\in[0,T]}\norm{\nabla\rho}_{L^2}^2\\
&\leq \varepsilon \sup_{t\in[0,T]}\cP(t)+C_\varepsilon\norm{\nabla u_0}_{L^2}^2,
\end{aligned}
\end{equation}
and 
\begin{equation}
\begin{aligned}
\abs{e^{\int_0^T h(t)\,dt}\int_0^T M'(t)e^{-\int_0^t h(s)\,ds}\,dt}&\leq \varepsilon\sup_{t\in[0,T]}\norm{\nabla v}_{L^2}^2+C_\varepsilon\sup_{t\in[0,T]}\norm{\nabla \rho}_{L^2}^2\\
&\leq \varepsilon \sup_{t\in[0,T]}\cP(t)+C_\varepsilon\norm{\nabla u_0}_{L^2}^2
\end{aligned}
\end{equation}
where $h(t)$ is an integrable function on $[0,\infty)$.

Finally, plugging \eqref{334} into \eqref{331}, we have
\begin{equation*}
\int_0^T\norm{\pi}_{H^1}^2\,dt\leq C\norm{\nabla u_0}_{L^2}^2,
\end{equation*}
which, together with \eqref{334}, completes the proof of \eqref{2.12}.
\end{proof}

Now, we turn back to prove Proposition \ref{prop3.1}.
\begin{proof}[Proof of Proposition \ref{prop3.1}]
Since, from Lemma \ref{lemma33} and the Sobolev embedding theorem,
\begin{equation*}
\sup_{t\in[0,T]}\norm{\nabla\rho}_{L^6}\leq C_1\sup_{t\in[0,T]}\norm{\Delta\rho}_{L^2}\leq C_1C\norm{\nabla u_0}_{L^2},
\end{equation*}
where $C$ as in Lemma \ref{lemma33} and $C_1$ is Sobolev embedding constant. Thus, if we choose
\begin{equation}\label{335}
\norm{\nabla u_0}_{L^2}\leq \delta_1:= (C_1C)^{-1},
\end{equation}
we can derive the first part of \eqref{32}. 

For the rest of \eqref{32}, using Lemma \ref{lemma33} again leads to
\begin{equation*}
\begin{aligned}
\int_0^T\left(\norm{\nabla v}_{L^2}^4+\norm{\Delta\rho}_{L^2}^4\right)\,dt&\leq \left(\sup_{t\in[0,T]}\norm{\Delta\rho}_{L^2}^2+\sup_{t\in[0,T]}\norm{\nabla v}_{L^2}^2\right)\int_0^T \left(\norm{\Delta\rho}^2_{L^2}+\norm{\nabla v}^2_{L^2}\right)\,dt\\
&\leq \lambda^{-1} C^2\norm{\nabla u_0}_{L^2}^4\\
&\leq \norm{\nabla u_0}_{L^2}^2
\end{aligned}
\end{equation*}
provided 
\begin{equation}\label{336}
\norm{\nabla u_0}_{L^2}\leq \delta_2:= \lambda^{1/2} C^{-1},
\end{equation}
where $C$ as in Lemma \ref{lemma33}.

It follows from \eqref{335} and \eqref{336} that one should choose $\delta:=\min\{1,\delta_1,\delta_2\}$. Of course, such $\delta$ depends only on $\Omega$, $c_0$, $\alpha$ and $\beta$ and, therefore, we estabished \eqref{32}.
\end{proof}

\subsubsection{Case \eqref{equation1.7}}\label{subsection3.2}
Similar with the preceding subsection, we are going to prove the following proposition. 
\begin{Proposition}\label{prop3.4}
There exists a positive constant $\hat\delta$ depending on $\Omega$, $c_0$, $\alpha$ and $\beta$ such that, if $\norm{\nabla u}_{L^2}\leq \hat\delta$ and 
\begin{equation}\label{3.39}
\sup_{t\in[0,T]}\norm{\nabla\rho}_{L^6}\leq 2,\quad\int_0^T\left(\norm{\nabla u}_{L^2}^4+\norm{\Delta\rho}_{L^2}^4\right)\,dt\leq 2\norm{\nabla u_0}^2_{L^2},
\end{equation}
then, one has
\begin{equation}\label{3.40}
\sup_{t\in[0,T]}\norm{\nabla\rho}_{L^6}\leq 1,\quad\int_0^T\left(\norm{\nabla u}_{L^2}^4+\norm{\Delta\rho}_{L^2}^4\right)\,dt\leq \norm{\nabla u_0}^2_{L^2}.
\end{equation}
\end{Proposition}

One should notice that the norms of $v$ and $u$ are equivalent in the following sense under condition \eqref{3.39},
\begin{equation}\label{342}
\begin{aligned}
\norm{v}_{L^p}+\norm{\nabla\rho}_{L^p}&\sim \norm{u}_{L^p}+\norm{\nabla\rho}_{L^p},\\
\norm{\nabla v}_{L^p}+\norm{\Delta\rho}_{L^p}+\norm{\nabla\rho}_{L^{2p}}^2&\sim \norm{\nabla u}_{L^p}+\norm{\Delta\rho}_{L^p}+\norm{\nabla\rho}_{L^{2p}}^2,\\
\norm{\Delta v}_{L^2}+\norm{\nabla\Delta\rho}_{L^2}&\sim\norm{\Delta u}_{L^2}+\norm{\nabla\Delta\rho}_{L^2},
\end{aligned}
\end{equation}
where $\eqref{342}_3$ is deduced by
\begin{equation*}
\begin{aligned}
\norm{\Delta v}_{L^2}&\leq C\left(\norm{\Delta u}_{L^2}+\norm{\nabla\Delta\rho}_{L^2}+\norm{\nabla\rho}_{L^6}\normf{\nabla^2\rho}_{L^3}+\norm{\nabla\rho}_{L^6}^3\right)\\
&\leq C\left(\norm{\Delta u}_{L^2}+\norm{\nabla\Delta\rho}_{L^2}+\norm{\nabla\rho}_{L^6}\norm{\nabla\Delta\rho}_{L^2}+\norm{\nabla\rho}_{L^6}^2\norm{\nabla\Delta\rho}_{L^2}\right)\\
&\leq C\left(\norm{\Delta u}_{L^2}+\norm{\nabla\Delta\rho}_{L^2}\right)
\end{aligned}
\end{equation*}
and vice versa.

Now, we come to prove. We can easily derive a similar lemma comparing with Lemma \ref{Lemma2.2} which is given as follows.
\begin{Lemma}\label{lemma3.5}
Let $(\rho,u,\pi)$ be a smooth solution of \eqref{equation1.1}, then there exist some positive constant $C$ depending only on $\Omega$, $c_0$, $\alpha$ and $\beta$ such that, for all $T\in (0,\infty)$,
\begin{equation}\label{343}
\sup_{t\in[0,T]}\norm{\rho-(\rho_0)_\Omega}_{L^2}^2+\int_0^T\norm{\nabla\rho}_{L^2}^2\,dt\leq C\norm{\nabla u_0}_{L^2}^2.
\end{equation}
Furthermore, if $\norm{\nabla u_0}_{L^2}\leq 1$ and the condition \eqref{3.39} holds, one has
\begin{gather}
\sup_{t\in[0,T]}\norm{\nabla\rho}_{L^2}^2+\int_0^T\left(\norm{\nabla\rho}_{L^3}^4+ \norm{\Delta\rho}_{L^2}^2\right)\,dt\leq C\norm{\nabla u_0}^2_{L^2},\label{344}\\
\sup_{t\in[0,T]}\cF(t)+\int_0^T\left(\cG(t)+\norm{\pi}_{H^1}^2\right)\,dt\leq C\norm{\nabla u_0}^2_{L^2},\label{3347}
\end{gather}
where
\begin{equation*}
\begin{gathered}
\cF(t):=\normf{\nabla u}_{L^2}^2+\norm{\Delta\rho}_{L^2}^2+\norm{\rho_t}_{L^2}^2,\\
\cG(t):= \norm{\nabla\Delta\rho}_{L^2}^2+\norm{u_t}_{L^2}^2+\norm{\Delta u}_{L^2}^2+\norm{\nabla\rho_t}_{L^2}^2
\end{gathered}
\end{equation*}
\end{Lemma}

\begin{proof}
\eqref{343} has been proved in Lemma \ref{Lemma2.2} and \eqref{344} is a trivial consequence of \eqref{33.8}. Indeed, using \eqref{342}, condition  \eqref{3.39}, Lemma \ref{Lemma221} and Poincar\'e's inequality leads to
$$
\begin{aligned}
\norm{\nabla v}_{L^2}^4&\leq C\left(\norm{\nabla u}_{L^2}^4+\norm{\Delta\rho}_{L^2}^4+\norm{\nabla\rho}_{L^3}^4\norm{\nabla\rho}_{L^6}^4\right)\\
&\leq C\left(\norm{\nabla u}_{L^2}^4+\norm{\Delta\rho}_{L^2}^4\right),
\end{aligned}
$$
and, thus,
\begin{equation*}
\left(\norm{\nabla\rho}_{L^2}^2\right)_t+\nu\norm{\Delta\rho}_{L^2}^2\leq C\left(\norm{\Delta\rho}_{L^2}^4+\norm{\nabla u}^4_{L^2}\right)\norm{\nabla\rho}_{L^2}^2.
\end{equation*}

To prove \eqref{3347}, we first come to get the lower order estimate of $u$. Multiplying $w$ on both sides of $\eqref{equation1.1}_2$ and integrating over $\Omega$, one has
\begin{equation}\label{346}
\begin{aligned}
\left(\int \frac{1}{2}\rho|u|^2\right)_t+\int 2\mu(\rho)|D(u)|^2&=\int \rho u_t\cdot Q +\int \rho u\cdot \nabla u\cdot Q-\int \dive[2\mu(\rho)D(u)]\cdot Q\\
&=\int \rho u_t\cdot Q +\int \rho u\cdot \nabla u\cdot Q+\int 2\mu(\rho)D(u)\cdot \nabla Q\\
&:=\sum_{i=1}^3 M_i,
\end{aligned}
\end{equation}
where, from Lemma \ref{Lemma221},
\begin{equation}\label{347}
\begin{cases}
|M_1|\leq C\norm{u_t}_{L^2}\norm{Q}_{L^2}\leq C_{\varepsilon_1}\norm{\nabla\rho}_{L^2}^2+\varepsilon_1\norm{u_t}_{L^2}^2,\\
|M_2|\leq C\norm{u}_{L^3}\norm{\nabla u}_{L^2}\norm{\nabla\rho}_{L^6}\leq C_{\varepsilon_2}\norm{\Delta\rho}_{L^2}^4\norm{u}_{L^2}^2+\varepsilon_2\norm{\nabla u}_{L^2},\\
|M_3|\leq C\norm{\nabla u}_{L^2}\norm{\nabla Q}_{L^2}\leq C_{\varepsilon_3}\norm{\Delta\rho}_{L^2}^2+\varepsilon_3\norm{\nabla u}_{L^2}.
\end{cases}
\end{equation}
Here, we have used the following control
\begin{equation*}
\begin{aligned}
\norm{\nabla Q}_{L^2}&\leq C\left(\norm{\Delta\rho}_{L^2}+\norm{\nabla\rho}_{L^4}^2\right)\leq C\left(\norm{\Delta\rho}_{L^2}+\norm{\nabla\rho}_{L^3}\norm{\nabla\rho}_{L^6}\right)\leq C\norm{\Delta\rho}_{L^2}.
\end{aligned}
\end{equation*}
Thus, substituting \eqref{347} into \eqref{346} gives
\begin{equation}\label{348}
\left(\norm{\sqrt\rho u}_{L^2}^2\right)_t+\nu\norm{\nabla u}_{L^2}^2\leq C_\varepsilon\norm{\Delta\rho}_{L^2}^4\norm{\sqrt\rho u}_{L^2}^2+C_\varepsilon\left(\norm{\nabla\rho}_{L^2}^2+\norm{\Delta\rho}_{L^2}^2\right)+\varepsilon\norm{u_t}_{L^2}^2.
\end{equation}

Multiplying $w_t$ on both sides of $\eqref{equation1.1}_2$ and integrating over $\Omega$, one has
\begin{equation}\label{349}
\begin{aligned}
\int\rho|u_t|^2+\left(\int\mu(\rho)|D(u)|^2\right)_t&=-\int\rho u_t\cdot Q_t-\int\rho u\cdot \nabla u\cdot w_t \\
&\quad+\int\mu(\rho)_t|D(u)|^2-\int\dive[2\mu(\rho)D(u)]\cdot Q_t\\
&:=\sum_{i=1}^4N_i,
\end{aligned}
\end{equation}
where, using Lemma \ref{Lemma221},
\begin{equation}\label{350}
\begin{cases}
|N_1|&\!\!\!\!\leq C\norm{u_t}_{L^2}\norm{Q_t}_{L^2}\leq C_{\varepsilon_1}\norm{\nabla\rho_t}_{L^2}^2+\varepsilon_1\norm{u_t}_{L^2}^2,\\
|N_2|&\!\!\!\!\leq C\norm{u}_{L^6}\norm{\nabla u}_{L^3}\norm{w_t}_{L^2}\\
&\!\!\!\!\leq C_{\varepsilon_2}\norm{\nabla u}_{L^2}^4\norm{\nabla u}_{L^2}^2+\varepsilon_2\left(\norm{\nabla\rho_t}_{L^2}^2+\norm{\Delta\rho}_{L^2}^4\norm{\rho_t}_{L^2}^2\right)\\
&\!\!\!\!\quad+\varepsilon_2\left(\norm{\Delta u}_{L^2}^2+\norm{u_t}_{L^2}^2\right),\\
|N_3|&\!\!\!\!\leq C\norm{\rho_t}_{L^3}\norm{\nabla u}_{L^2}\norm{\nabla u}_{L^6}\leq C_{\varepsilon_3}\norm{\nabla u}_{L^2}^4\norm{\rho_t}_{L^2}^2+C\norm{\nabla\rho_t}_{L^2}^2+\varepsilon_3\norm{\Delta u}_{L^2}^2,\\
|N_4|&\!\!\!\!\leq C\left(\norm{\nabla\rho}_{L^6}\norm{\nabla u}_{L^3}+\norm{\Delta u}_{L^2}\right)\norm{Q_t}_{L^2}\\
&\!\!\!\!\leq C_{\varepsilon_4}\norm{\Delta\rho}_{L^2}^4\norm{\nabla u}_{L^2}^2+C_{\varepsilon_4}\norm{\nabla\rho_t}_{L^2}^2+\varepsilon_4\norm{\Delta u}_{L^2}^2,
\end{cases}
\end{equation}
where we have used 
$$\norm{Q_t}_{L^2}\leq C\left(\norm{\nabla\rho_t}_{L^2}+\norm{\nabla\rho}_{L^6}\norm{\rho_t}_{L^3}\right)\leq C\norm{\nabla\rho_t}_{L^2}.$$
Combining \eqref{349} and \eqref{350} leads to
\begin{equation}\label{351}
\begin{aligned}
\left(\normf{\sqrt{\mu(\rho)}|D(u)|}_{L^2}^2\right)_t+\nu\norm{u_t}_{L^2}^2&\leq C_\varepsilon\left(\norm{\nabla u}_{L^2}^4+\norm{\Delta \rho}_{L^2}^4\right)\left(\norm{\nabla u}_{L^2}^2+\norm{\rho_t}_{L^2}^2\right)\\
&\quad+C_\varepsilon\norm{\nabla\rho_t}_{L^2}^2+\varepsilon\norm{\Delta u}_{L^2}^2.
\end{aligned}
\end{equation}

To get $\norm{\Delta u}_{L^2}$, we follow the proof \eqref{328}--\eqref{331} and use Lemma \ref{lemma26} $(2)$ with $\Phi=-c_0\nabla\rho^{-1}$ and condition \eqref{3.39} to deduce
\begin{align*}
\norm{v}^2_{H^2}+\norm{\nabla \pi}^2_{L^2}&\leq C\left[\norm{\nabla\rho}^4_{L^6}\norm{\nabla v}^2_{L^2}+\left(1+\norm{\nabla\rho}^4_{L^6}\right)\left(\norm{F}_{L^2}^2+\norm{\Phi}_{H^2}^2\right)\right]\\
&\leq C\left(\normf{F}^2_{L^2}+\norm{\Phi}^2_{H^2}+\norm{\Delta\rho}^4_{L^2}\norm{\nabla v}^2_{L^2}\right)\\
&\leq C\left(\norm{v_t}^2_{L^2}+\norm{\nabla\rho_t}^2_{L^2}\right)+C\left(\norm{\Delta\rho}_{L^2}^4+\norm{\nabla v}_{L^2}^4\right)\norm{\Delta\rho}_{L^2}^2+C\norm{\nabla\Delta\rho}_{L^2}^2\\
&\quad+C\norm{\nabla v}_{L^2}^2+C_{\varepsilon}\left(\norm{\nabla v}_{L^2}^4+\norm{\Delta\rho}_{L^2}^4\right)\norm{\nabla v}_{L^2}^2+\varepsilon\norm{v}_{H^2}^2
\end{align*}
where $F$ as in \eqref{F}--\eqref{329}. Thus, we still have \eqref{331} and, if we convert $v$ into $u$ and $\rho$ by \eqref{342} and condition \eqref{3.39}, we can derive the bounds for $\Delta u$, that is,
\begin{equation}\label{352}
\begin{aligned}
\norm{\Delta u}_{L^2}^2+\norm{\pi}^2_{H^1}&\leq C\left(\norm{\Delta\rho}_{L^2}^4+\norm{\nabla u}_{L^2}^4\right)\left(\norm{\Delta\rho}_{L^2}^2+\norm{\nabla u}_{L^2}^2\right)\\
&\quad+C\left(\norm{\nabla\Delta\rho}_{L^2}^2+\norm{\nabla u}_{L^2}^2+\norm{u_t}^2_{L^2}+\norm{\nabla\rho_t}^2_{L^2}\right)
\end{aligned}
\end{equation}

Combining \eqref{348}, \eqref{351} and \eqref{352} and choosing $\varepsilon$ small enough, one has
\begin{equation}\label{3354}
\begin{aligned}
&\left(\normf{\nabla u}_{L^2}^2\right)_t+\frac{\nu}{2}\norm{\nabla u}_{L^2}^2+\frac{\varepsilon}{2C}\norm{\Delta u}_{L^2}^2+\frac{\nu}{2}\norm{u_t}_{L^2}^2\\
&\leq C_\varepsilon\left(\norm{\nabla u}_{L^2}^4+\norm{\Delta \rho}_{L^2}^4\right)\cF(t)+C_\varepsilon\norm{\nabla\rho_t}_{L^2}^2+\varepsilon\norm{\nabla\Delta\rho}_{L^2}^2+C\left(\norm{\nabla\rho}_{L^2}^2+\norm{\Delta\rho}_{L^2}^2\right),
\end{aligned}
\end{equation}
where we have used
$$\norm{\sqrt\rho u}_{L^2}+\normf{\sqrt{\mu(\rho)}|D(u)|}_{L^2}\sim \norm{\nabla u}_{L^2}.$$

Similarly, converting $v$ into $u$ and $\rho$, we can also obtain an analogous estimates from \eqref{318}, that is,
\begin{equation*}
\begin{aligned}
\left(\norm{\Delta\rho}_{L^2}^2\right)_t+\nu \norm{\nabla\Delta\rho}_{L^2}^2&\leq C\left(\norm{\nabla u}_{L^2}^4+\norm{\Delta\rho}_{L^2}^4\right)\norm{\Delta\rho}_{L^2}^2\\
&\quad+C_{\varepsilon_1}\norm{\Delta\rho}_{L^2}^4\norm{\nabla u}_{L^2}^2+\varepsilon_1\norm{\Delta u}_{L^2}^2,
\end{aligned}
\end{equation*}
which, combining with \eqref{352}, gives
\begin{equation}\label{355}
\begin{aligned}
\left(\norm{\Delta\rho}_{L^2}^2\right)_t+\nu \norm{\nabla\Delta\rho}_{L^2}^2&\leq C_{\varepsilon_1}\left(\norm{\nabla u}_{L^2}^4+\norm{\Delta\rho}_{L^2}^4\right)\left(\norm{\Delta\rho}_{L^2}^2+\norm{\nabla u}_{L^2}^2\right)\\
&\quad+{\varepsilon_1}\left(\norm{\nabla u}_{L^2}^2+\norm{\nabla \rho_t}_{L^2}^2\right).
\end{aligned}
\end{equation}
Combining \eqref{3354}--\eqref{355} and letting $\varepsilon,\,\varepsilon_1$ suitably small yield that, $\exists\,\nu>0$,
\begin{equation}\label{3356}
\begin{aligned}
&\left(\normf{\nabla u}_{L^2}^2+\norm{\Delta\rho}_{L^2}^2\right)_t+\nu\left(\norm{\nabla\Delta\rho}_{L^2}^2+\norm{\nabla u}_{L^2}^2+\norm{\Delta u}_{L^2}^2+\norm{u_t}_{L^2}^2\right)\\
&\leq C\left(\norm{\nabla u}_{L^2}^4+\norm{\Delta \rho}_{L^2}^4\right)\cF(t)+C\norm{\nabla\rho_t}_{L^2}^2+C\left(\norm{\nabla\rho}_{L^2}^2+\norm{\Delta\rho}_{L^2}^2\right),
\end{aligned}
\end{equation}

On the other hand, from \eqref{33.6}, we can deduce similarly that
\begin{equation}\label{356}
\left(\norm{\rho_t}_{L^2}^2\right)_t+\nu\norm{\nabla\rho_t}_{L^2}^2\leq C_{\varepsilon_2}\norm{\Delta\rho}_{L^2}^4\norm{\rho_t}_{L^2}^2+C_{\varepsilon_2}\norm{\nabla\rho}_{L^2}^2+\varepsilon_2\norm{u_t}_{L^2}^2.
\end{equation}
Then times $2C$ for \eqref{356} and plugging it into \eqref{3356}, choosing $\varepsilon_2$ sufficiently small and using Poincar\'e's inequality, we have, for some positive constant $\nu$,
\begin{equation}\label{3361}
\cF'(t)+\nu\cG(t)\leq C\left(\norm{\nabla u}_{L^2}^4+\norm{\Delta\rho}_{L^2}^4\right)\cF(t)+C\left(\norm{\nabla\rho}_{L^2}^2+\norm{\Delta\rho}_{L^2}^2\right),
\end{equation}
where we have used the following equivalent norms for convenience
\begin{equation*}
\begin{gathered}
\cF(t)\sim \normf{\nabla u}_{L^2}^2+\norm{\Delta\rho}_{L^2}^2+2C\norm{\rho_t}_{L^2}^2,\\
\cG(t)\sim \norm{\nabla\Delta\rho}_{L^2}^2+\norm{u_t}_{L^2}^2+\norm{\Delta u}_{L^2}^2+2C\norm{\nabla\rho_t}_{L^2}^2
\end{gathered}
\end{equation*}
and these equivalences do not have an influence on the final result after applying the Gr$\mathrm{\ddot{o}}$nwall's inequality. 

Thus, we get the higher order estimates for $(\rho, u)$ by using Gr$\mathrm{\ddot{o}}$nwall's inequality and \eqref{343}--\eqref{344} and the estimate of $\pi$ can be obtained from \eqref{352}. Consequently, we show the estimate \eqref{3347} and finished the proof.
\end{proof}

\begin{proof}[Proof of Proposition \ref{prop3.4}]
The proof of Proposition \ref{prop3.4} is exactly same with that of Proposition \ref{prop3.1} and, thus, we omit the proof and leave it proof to readers.
\end{proof}

\subsection{Proof of Theorem \ref{Theorem1.1}}\label{P12}
With the uniform bounds hold in our hand, the proof is rather simple. We first come to prove the case \eqref{equation1.6}. Using Lemma \ref{local}, there exists a unique strong solution $(\rho,u)$ of \eqref{equation1.1} on $\Omega\times(0,T_1)$ with initial data $(\rho_0,u_0)$ satifying the boundary condition \eqref{equation1.6}, for some positive time $T_1$. Then, one may use the a priori estimates, Proposition \eqref{prop3.1} and Lemma \ref{32}--\ref{lemma33} to extend the strong solution $(\rho,u)$ globally in time. Indeed, if $T_1<\infty$ is the maximal time for existence, then using the uniform bounds, we have
\begin{equation}
(\rho,u)(x,T_1):=\lim_{t\to T_1^-}(\rho,u)(x,t)\text{ in the sense of }H^2\times H^1
\end{equation}
satisfying the conditions imposed on the initial data, that is, $\alpha\leq\rho(T_1)\leq \beta$ and $u(T_1)\in H^1_\omega$, at the time $T_1$. Furthermore, it is easy to check that $(\rho,u)(x,T_1)$ satisfies the compatiablity condition \eqref{equation1.3}. Therefore, we can take $(\rho,u)(x,T_1)$ as the initial data and apply Lemma \ref{local} to extend the strong solution beyond $T_1$. This contradicts the maximality of $T_1$ and, hence, we finish the proof of Theorem \ref{Theorem1.1} for the case \eqref{equation1.6}.

However, for $(\rho,u)$ satisfying \eqref{equation1.7}, we can use Lemma \ref{local} and Remark \ref{local2} to extend $(\rho,u)$ on $\Omega\times(0,T_1)$ to the global one for every fixed $\epsilon\in (0,1]$. Then, using the a priori estimates, Proposition \ref{prop3.4} and Lemma \ref{lemma3.5}, we can get a uniform bounds for $(\rho^\epsilon,u^\epsilon,\pi^\epsilon)$, for all $\epsilon\in (0,1]$. More precisely, one may has, as $\epsilon\to 0^+$,
\begin{equation}\label{equation363}
\begin{cases}
\rho^\epsilon\wsconverge \rho\quad \text{in }C([0,T];H^2)\cap L^2(0,T;H^3),\\
\rho_t^\epsilon\wsconverge \rho_t\quad \text{in }C([0,T];L^2)\cap L^2(0,T;H^1),\\
u^\epsilon\wsconverge u\quad \text{in }C([0,T];H^1)\cap L^2(0,T;H^2),\\
u_t^\epsilon\wsconverge u_t\quad \text{in }C([0,T];H^1)\cap L^2(0,T;L^2),\\
\pi^\epsilon\wconverge \pi\quad \text{in } L^2(0,T;H^1).\\
\end{cases}
\end{equation}
Then, after applying Lemma \ref{lemma221}, we may derive that
\begin{equation}\label{equation364}
\begin{cases}
\rho^\epsilon\longrightarrow\rho \text{ uniformly for all }(x,t)\in \overline\Omega\times[0,T],\\
\rho^\epsilon\sconverge \rho \quad\text{in }C([0,T];H^2),\\
u^\epsilon\sconverge u\quad\text{in }C([0,T];H^1).
\end{cases}
\end{equation}
\eqref{equation363} and \eqref{equation364} are eough to let $\epsilon\to 0^+$ and recover to the original system \eqref{equation1.1}. The uniqueness can be obtained by similar method in \cite{zjw}.

\section{Proof of Theorem \ref{Theorem1.3}}\label{section4}

We first come to prove the blowup criterion. Throughout this section, we let $(\rho,u,\pi)$ be a strong solution described in Theorem \ref{Theorem1.3} and $\tilde C$ be a positive generic constant depending on $c_0$, $\alpha$, $\beta$, $T^*$, $M_0$ and $\norm{u_0}_{H^1}$. Suppose that \eqref{serrin} were false, that is, for some $r$ and $s$,
\begin{equation}\label{4.1}
\lim_{T\to T^*}\norm{u}_{L^s(0,T;L^r)}\leq M_0<\infty,
\end{equation}
or, equivalently,
\begin{equation*}
\lim_{T\to T^*}\left(\norm{v}_{L^s(0,T;L^r)}+\norm{\nabla\rho}_{L^s(0,T;L^r)}\right)\leq M_0,
\end{equation*}
we want to show the following estimate holds.

\begin{Proposition}\label{prop4.1}
Under the above condition, one has, for all $T\in[0,T^*)$,
\begin{equation}
\sup_{t\in[0,T]}\left(\norm{\rho_t}_{L^2}^2+\norm{\rho}^2_{H^2}+\norm{u}^2_{H^1}\right)+\int_0^T\left(\norm{\rho_t}_{H^1}^2+\norm{\rho}_{H^3}^2+\norm{\nabla u}_{H^1}^2\right)\,dt\leq \tilde C.
\end{equation}
\end{Proposition}

The proof of Proposition \ref{prop4.1} will be separated into the following two parts. 

\subsection{Case for $(\rho,u)$ satisfying \eqref{equation1.6}}\label{11}

The first lemma is the part of Lemma \ref{Lemma2.2}, we give it here for convenience.
\begin{Lemma}\label{lemma4.2}
The following bounds hold for condition \eqref{equation1.6} and for all $T\in[0,T^*)$, that is,
\begin{equation}
\alpha\leq\rho\leq\beta,\quad\sup_{t\in[0,T]}\norm{\rho-(\rho_0)_\Omega}_{L^2}^2+\nu\int_0^T\norm{\nabla\rho}_{L^2}^2\,dt\leq \norm{\rho_0-(\rho_0)_\Omega}_{L^2}^2,
\end{equation}
\end{Lemma}

Next, we give the lower order bounds for $(\log\rho,v)$, that is,
\begin{Lemma}\label{lemma44}
Suppose that \eqref{4.1} holds and $(\rho,u)$ satisfies \eqref{equation1.6}, then one has
\begin{equation}\label{equation4.4}
\sup_{t\in[0,T]}\left(\norm{\nabla\log\rho}^2_{L^2}+\norm{v}^2_{L^2}\right)+\int_0^T\left(\norm{\Delta\log\rho}_{L^2}^2+\norm{\nabla v}_{L^2}^2\right)\,dt\leq \tilde C.
\end{equation}
\end{Lemma}
\begin{proof}
We first change \eqref{equation1.18} to the form
\begin{equation}\label{log}
(\log\rho)_t +v\cdot\nabla\log\rho  - c_0\rho^{-1}\Delta\log\rho=0,
\end{equation}
and, then, multiplying $-\Delta\log\rho$ on both sides of \eqref{log}, integrating over $\Omega$ and using Lemma \ref{Lemma221} imply that
\begin{equation}\label{44.6}
\begin{aligned}
&\left(\frac{1}{2}\int |\nabla\log\rho|^2\right)_t+\int c_0\rho^{-1}|\Delta\log\rho|^2\\
&=\int (v\cdot \nabla\log\rho)\Delta\log\rho\\
&\leq \norm{\nabla\log\rho}_{L^r}\norm{v}_{L^{\frac{2r}{r-2}}}\norm{\Delta\log\rho}_{L^2}\\
&\leq C_\varepsilon\norm{\nabla\log\rho}_{L^r}^2\norm{v}_{L^2}^{\frac{2r-6}{r}}\norm{\nabla v}_{L^2}^{\frac{6}{r}}+\varepsilon\norm{\Delta\log\rho}_{L^2}^2\\
&\leq C_\varepsilon\left(\norm{\nabla\rho}_{L^r}^s+1\right)\norm{v}^2_{L^2}+\varepsilon\left(\norm{\Delta\log\rho}_{L^2}^2+\norm{\nabla v}_{L^2}^2\right),
\end{aligned}
\end{equation}
that is
\begin{equation}\label{equation4.6}
\left(\norm{\nabla\log\rho}_{L^2}^2\right)_t+\nu\norm{\Delta\log\rho}_{L^2}^2\leq C_\varepsilon\left(\norm{\nabla\rho}_{L^r}^s+1\right)\norm{v}^2_{L^2}+\varepsilon\norm{\nabla v}_{L^2}^2.
\end{equation}

To estimate the rest part of \eqref{equation4.4}, it follows from \eqref{218}--\eqref{219} that
\begin{equation}\label{equation4.8}
\begin{aligned}
\left(\norm{\sqrt{\rho}v}_{L^2}^2\right)_t+\nu\norm{\nabla v}_{L^2}^2&\leq C\left(\norm{\nabla\log\rho}_{H^1}\norm{\nabla v}_{L^2}+\norm{\nabla\log\rho}_{L^r}\norm{v}_{L^{\frac{2r}{r-2}}}\norm{\nabla v}_{L^2}\right)\\
&\quad+\sum_{i=1}^4H_i,
\end{aligned}
\end{equation}
For $H_1$--$H_4$, using Lemma \ref{Lemma221}, we have, from \eqref{220}--\eqref{221}
\begin{equation}\label{equation4.9}
\begin{cases}
|H_1|&\!\!\!\!\leq C\left(\norm{\Delta\log\rho}_{L^2}+\norm{\nabla\log\rho}_{L^r}\norm{\nabla\log\rho}_{L^{\frac{2r}{r-2}}}\right)\norm{\nabla v}_{L^2}\\
&\!\!\!\!\leq C_{\varepsilon_1}\left(\norm{\nabla\rho}_{L^r}^s+1\right)\norm{\nabla\log\rho}_{L^2}^2+\varepsilon_1\left(\norm{\nabla v}_{L^2}^2+\norm{\Delta\log\rho}_{L^2}^2\right),\\
|H_2|&\!\!\!\!\leq C\norm{\nabla\log\rho}_{L^r}\norm{v}_{L^{\frac{2r}{r-2}}}\norm{\nabla v}_{L^2}\\
&\!\!\!\!\leq C_{\varepsilon_2}\left(\norm{\nabla\rho}_{L^r}^s+1\right)\norm{v}_{L^2}^2+\varepsilon_2\left(\norm{\nabla v}_{L^2}^2+\norm{\Delta\log\rho}_{L^2}^2\right),\\
|H_3|&\!\!\!\!\leq C\norm{\nabla\log\rho}_{L^r}\norm{\nabla\log\rho}_{L^{\frac{2r}{r-2}}}\norm{\nabla v}_{L^2}\\
&\!\!\!\!\leq C_{\varepsilon_3}\left(\norm{\nabla\rho}_{L^r}^s+1\right)\norm{\nabla\log\rho}_{L^2}^2+\varepsilon_3\left(\norm{\nabla v}_{L^2}^2+\norm{\Delta\log\rho}_{L^2}^2\right),\\
|H_4|&\!\!\!\!\leq C\left(\norm{v}_{H^1}\norm{\nabla\log\rho}_{H^1}+\norm{\nabla\log\rho}_{L^r}\norm{v}_{L^{\frac{2r}{r-2}}}\norm{\nabla\log\rho}_{L^2}\right)\\
&\!\!\!\!\leq C_{\varepsilon_4}\left[\norm{\Delta\log\rho}_{L^2}^2+\left(\norm{\nabla\rho}_{L^r}^s+1\right)\norm{v}_{L^2}^2\right]+\varepsilon_4\left(\norm{\nabla v}_{L^2}^2+\norm{\nabla\log\rho}_{L^2}^2\right).
\end{cases}
\end{equation}
Combining \eqref{equation4.8} and \eqref{equation4.9}, we deduce that
\begin{equation}\label{equation4.10}
\left(\norm{\sqrt{\rho}v}_{L^2}^2\right)_t+\nu\norm{\nabla v}_{L^2}^2\leq C\left(\norm{\nabla\rho}_{L^r}^s+1\right)\left(\norm{v}_{L^2}^2+\norm{\nabla\log\rho}_{L^2}^2\right)+C\norm{\Delta\log\rho}_{L^2}^2.
\end{equation}
Multiplying $2\varepsilon$ on \eqref{equation4.10} and alonging with \eqref{equation4.6}, then chooseing $\varepsilon$ suitably small gives
\begin{equation*}
\begin{aligned}
&\left(\frac{1}{2C}\norm{\sqrt{\rho}v}_{L^2}^2+\norm{\nabla\log\rho}_{L^2}^2\right)_t+\frac{\nu}{2}\left(\frac{1}{2C}\norm{\nabla v}_{L^2}^2+\norm{\Delta\log\rho}_{L^2}^2\right)\\
&\leq C\left(\norm{\nabla\rho}_{L^r}^s+1\right)\left(\norm{v}_{L^2}^2+\norm{\nabla\log\rho}_{L^2}^2\right).
\end{aligned}
\end{equation*}
which, using Gr$\mathrm{\ddot{o}}$nwall's inequality, condition \eqref{4.1} and Lemma \ref{lemma4.2}, implies \eqref{equation4.4}. 
\end{proof}

\begin{Lemma}\label{lemma45}
Suppose that \eqref{4.1} holds and $(\rho,u)$ satisfies \eqref{equation1.6}, then
\begin{equation}\label{4.11}
\begin{aligned}
\sup_{t\in[0,T]}\tilde\cP(t)+\int_0^T\left(\tilde\cQ(t)+\norm{\pi}_{H^1}^2\right)\,dt\leq \tilde C.
\end{aligned}
\end{equation}
where
\begin{gather*}
\tilde\cP(t):=\norm{(\log\rho)_t}_{L^2}^2+\norm{\Delta\rho}^2_{L^2}+\norm{\Delta\log\rho}^2_{L^2}+\norm{\nabla v}^2_{L^2},\\
\tilde\cQ(t):=\norm{\nabla(\log\rho)_t}_{L^2}^2+\norm{\nabla\Delta\rho}_{L^2}^2+\norm{\nabla\Delta\log\rho}_{L^2}^2+\normf{\nabla^2 v}_{L^2}^.
\end{gather*}
\end{Lemma}

\begin{proof}
Applying $-\nabla\Delta\log\rho\cdot\nabla$ on both sides of \eqref{log} and integrating over $\Omega$, we have
\begin{equation}\label{44.12}
\begin{aligned}
\left(\int \frac{1}{2}|\Delta\log\rho|^2\right)_t+\int c_0\rho^{-1}|\nabla\Delta\log\rho|^2&=\int \nabla\Delta\log\rho\cdot \nabla v\cdot \nabla\log\rho\\
&\quad+\int c_0\rho^{-1}\Delta\log\rho\nabla\log\rho\cdot\nabla\Delta\log\rho\\
&\quad+\int v\cdot \nabla^2\log\rho\cdot\nabla\Delta\log\rho\\
&:=\sum_{i=1}^3 O_i,
\end{aligned}
\end{equation}
where, for terms $O_1$ and $O_2$, we use Lemma \ref{Lemma221} to get
\begin{equation}\label{4.13}
\begin{cases}
|O_1|&\!\!\!\!\leq \norm{\nabla\log\rho}_{L^r}\norm{\nabla v}_{L^{\frac{2r}{r-2}}}\norm{\nabla\Delta\log\rho}_{L^2}\\
&\!\!\!\!\leq C_{\varepsilon_1}\left(\norm{\nabla\rho}_{L^r}^s+1\right)\norm{\nabla v}_{L^2}^2+\varepsilon_1\left(\norm{\nabla\Delta\log\rho}_{L^2}^2+\normf{v}_{H^2}^2\right),\\
|O_2|&\!\!\!\!\leq C\norm{\nabla\log\rho}_{L^r}\norm{\Delta\log\rho}_{L^{\frac{2r}{r-2}}}\norm{\nabla\Delta\log\rho}_{L^2}\\
&\!\!\!\!\leq C_{\varepsilon_2}\left(\norm{\nabla\rho}_{L^r}^s+1\right)\norm{\Delta\log\rho}_{L^2}^2+\varepsilon_2\norm{\nabla\Delta\rho}_{L^2}^2.
\end{cases}
\end{equation}
For $O_3$, we integrate by parts to get
\begin{equation}
\begin{aligned}\label{4.14}
O_3&=\int v_i\partial_{ij}\log\rho\partial_j\Delta\log\rho\\
&=\int_\partial (v_i\partial_{ij}\log\rho n_j)\Delta\log\rho-\int (\partial_jv_i\partial_{ij}\log\rho)\Delta\log\rho\\
&=\int_\partial (v_i\partial_{ij}\log\rho n_j)\Delta\log\rho-\int_\partial (\partial_jv_i\partial_{j}\log\rho n_i)\Delta\log\rho+\int \partial_jv_i\partial_{j}\log\rho\partial_i\Delta\log\rho\\
&:=B_1+B_2+B_3.
\end{aligned}
\end{equation}
Since the simplest part $B_3$ can be handled similarly like $O_1$, we only need estimate $B_1$ and $B_2$. First, using the boundary condition $v\cdot n=n\cdot \nabla\log\rho=0$ and Lemma \ref{Lemma221}, we have
\begin{equation}
\begin{aligned}
B_1&=\int_\partial v\cdot \nabla^2\log\rho\cdot n\Delta\log\rho\\
&=-\int_\partial v\cdot\nabla n \cdot\nabla\log\rho \Delta\log\rho\\
&=\int_\partial (n\times v^\perp)\cdot\nabla n \cdot\nabla\log\rho \Delta\log\rho\\
&=\int (\curle v^\perp\cdot\nabla n \cdot\nabla\log\rho)\Delta\log\rho-\int v^\perp\cdot\left[\nabla\Delta\log\rho\times(\nabla n \cdot\nabla\log\rho)\right] \\
&\quad-\int v^\perp\cdot\curle(\nabla n \cdot\nabla\log\rho) \Delta\log\rho\\
&\leq C\norm{\nabla\log\rho}_{L^r}\norm{\nabla v}_{L^{\frac{2r}{r-2}}}\norm{\Delta\log\rho}_{L^2}+C\norm{\nabla\log\rho}_{L^r}\norm{v}_{L^{\frac{2r}{r-2}}}\norm{\nabla\Delta\log\rho}_{L^2}\\
&\quad+C\norm{v}_{L^3}\norm{\Delta\log\rho}_{L^6}\norm{\Delta\log\rho}_{L^2}\\
&\leq C_{\varepsilon_3}(\norm{\nabla\rho}_{L^r}^s+1)\norm{\nabla v}_{L^2}^2+C_{\varepsilon_3}\norm{v}^4_{L^3}\norm{\Delta\log\rho}^2_{L^2}+\varepsilon_3\norm{\nabla\Delta\log\rho}^2_{L^2}.
\end{aligned}
\end{equation}
Hence,
\begin{equation}\label{4.16}
\begin{aligned}
|B_1|\leq C_{\varepsilon_3}(\norm{\nabla\rho}_{L^r}^s+1)\norm{\nabla v}_{L^2}^2+C_{\varepsilon_3}\norm{v}^4_{L^3}\norm{\Delta\log\rho}^2_{L^2}+\varepsilon_3\norm{\nabla\Delta\log\rho}^2_{L^2}.
\end{aligned}
\end{equation}

Similarly, for $B_2$, one has
\begin{equation}
\begin{aligned}
B_2&=-\int_\partial \nabla\log\rho\cdot \nabla v\cdot n\Delta\log\rho\\
&=\int_\partial \nabla\log\rho\cdot \nabla n\cdot v\Delta\log\rho\\
&=\int_\partial \nabla\log\rho\cdot \nabla n\cdot (v^\perp \times n)\Delta\log\rho\\
&=-\int (\nabla\log\rho\cdot \nabla n\cdot \curle v^\perp) \Delta\log\rho+\int \nabla\Delta\log\rho\times(\nabla\log\rho\cdot \nabla n)\cdot v^\perp\\
&\quad+\int \Delta\log\rho\curle(\nabla\log\rho\cdot \nabla n)\cdot v^\perp\\
&\leq C_{\varepsilon_4}(\norm{\nabla\rho}_{L^r}^s+1)\norm{\nabla v}_{L^2}^2+C_{\varepsilon_4}\norm{v}^4_{L^3}\norm{\Delta\log\rho}^2_{L^2}+\varepsilon_4\norm{\nabla\Delta\log\rho}^2_{L^2},
\end{aligned}
\end{equation}
that is,
\begin{equation}\label{4.18}
|B_2|\leq C_{\varepsilon_4}(\norm{\nabla\rho}_{L^r}^s+1)\norm{\nabla v}_{L^2}^2+C_{\varepsilon_4}\norm{v}^4_{L^3}\norm{\Delta\log\rho}^2_{L^2}+\varepsilon_4\norm{\nabla\Delta\log\rho}^2_{L^2}.
\end{equation}
Combining \eqref{4.13}--\eqref{4.14}, \eqref{4.16} and \eqref{4.18}, one can deduce that
\begin{equation}\label{4.19}
\begin{aligned}
&\left(\norm{\Delta\log\rho}_{L^2}^2\right)_t+\nu\norm{\nabla\Delta\log\rho}_{L^2}^2\\
&\leq C\left(\norm{\nabla\rho}_{L^r}^s+\norm{v}^4_{L^3}+1\right)\norm{\Delta\log\rho}_{L^2}^2+C_\varepsilon\left(\norm{\nabla\rho}_{L^r}^s+1\right)\norm{\nabla v}_{L^2}^2+\varepsilon\norm{v}_{H^2}^2.
\end{aligned}
\end{equation}

On the other hand, we slightly change \eqref{2.20} (more precisely, $I_3$) into the form
\begin{equation}
\begin{aligned}
\left(\int \frac{1}{2}|\Delta\rho|^2\right)_t+\int c_0\rho^{-1}|\nabla\Delta\rho|^2&=\int \nabla\Delta\rho\cdot \nabla v\cdot \nabla\rho+\int v\cdot \nabla^2\rho\cdot\nabla\Delta\rho\\
&\quad+c_0\int \nabla|\nabla\log\rho|^2\cdot\nabla\Delta\rho+\int c_0\rho^{-2}\Delta\rho\nabla\rho\cdot\nabla\Delta\rho\\
&:=\sum_{i=1}^4 I_i.
\end{aligned}
\end{equation}
Then, exactly following the proof of \eqref{4.13}--\eqref{4.18}, we can obtain the festimate which is similar with \eqref{4.19}, that is,
\begin{equation}\label{44.21}
\begin{aligned}
\left(\norm{\Delta\rho}_{L^2}^2\right)_t+\nu\norm{\nabla\Delta\rho}_{L^2}^2&\leq C_\varepsilon\left(\norm{\nabla\rho}_{L^r}^s+\norm{v}_{L^3}^4+1\right)\left(\norm{\Delta\rho}_{L^2}^2+\norm{\Delta\log\rho}_{L^2}^2\right)\\
&\quad+C_\varepsilon\left(\norm{\nabla\rho}_{L^r}^s+1\right)\norm{\nabla v}_{L^2}^2+\varepsilon\left(\norm{v}_{H^2}^2+\norm{\nabla\Delta\log\rho}_{L^2}^2\right),
\end{aligned}
\end{equation}
together with \eqref{4.19} yields
\begin{equation}\label{4.22}
\begin{aligned}
&\left(\norm{\Delta\rho}_{L^2}^2+\norm{\Delta\log\rho}_{L^2}^2\right)_t+\nu\left(\norm{\nabla\Delta\rho}_{L^2}^2+\norm{\nabla\Delta\log\rho}_{L^2}^2\right)\\
&\leq C\left(\norm{\nabla\rho}_{L^r}^s+\norm{v}_{L^3}^4+1\right)\left(\norm{\Delta\rho}_{L^2}^2+\norm{\Delta\log\rho}_{L^2}^2+\norm{\nabla v}_{L^2}^2\right)+\varepsilon\norm{v}_{H^2}^2,
\end{aligned}
\end{equation}

For the estimate of $(\log\rho)_t$, applying $(\log\rho)_t\partial_t$ on both sides of \eqref{log} and integrating over $\Omega$, one has
\begin{equation}\label{4.20}
\begin{aligned}
&\left(\frac{1}{2}\int|(\log\rho)_t|^2\right)_t+\int c_0\rho^{-1}|\nabla(\log\rho)_t|^2\\
&=-\int c_0\rho^{-1}\nabla(\log\rho)_t\cdot\nabla\log\rho(\log\rho)_t-\int v_t\cdot\nabla\log\rho(\log\rho)_t+\int c_0\rho^{-1}|(\log\rho)_t|^2|\nabla\log\rho|^2\\
&:=\sum_{i=1}^3P_i,
\end{aligned}
\end{equation}
where, using Lemma \ref{Lemma221},
\begin{equation}\label{4.21}
\begin{cases}
|P_1|&\!\!\!\!\leq C\norm{\nabla\log\rho}_{L^r}\norm{(\log\rho)_t}_{L^{\frac{2r}{r-2}}}\norm{\nabla(\log\rho)_t}_{L^2}\\
&\!\!\!\!\leq C_{\varepsilon_1}(\norm{\nabla\rho}_{L^r}^s+1)\norm{(\log\rho)_t}_{L^2}^2+\varepsilon_1\norm{\nabla(\log\rho)_t}_{L^2}^2\\
|P_2|&\!\!\!\!\leq \norm{\nabla\log\rho}_{L^r}\norm{(\log\rho)_t}_{L^{\frac{2r}{r-2}}}\norm{v_t}_{L^2}\\
&\!\!\!\!\leq C_{\varepsilon_2}(\norm{\nabla\rho}_{L^r}^s+1)\norm{(\log\rho)_t}_{L^2}^2+\varepsilon_2\norm{v_t}_{L^2}^2\\
|P_3|&\!\!\!\!\leq \norm{\nabla\log\rho}^2_{L^r}\norm{(\log\rho)_t}^2_{L^{\frac{2r}{r-2}}}\\
&\!\!\!\!\leq C_{\varepsilon_3}(\norm{\nabla\rho}_{L^r}^s+1)\norm{(\log\rho)_t}_{L^2}^2+\varepsilon_3\norm{\nabla(\log\rho)_t}_{L^2}^2.
\end{cases}
\end{equation}
Combining \eqref{4.20} and \eqref{4.21} leads to
\begin{equation}\label{4.25}
\left(\norm{(\log\rho)_t}_{L^2}^2\right)_t+\nu\norm{\nabla(\log\rho)_t}_{L^2}^2\leq C_\varepsilon(\norm{\nabla\rho}_{L^r}^s+1)\norm{(\log\rho)_t}_{L^2}^2+\varepsilon\norm{v_t}_{L^2}^2
\end{equation}

We still need to treat the higher order bounds for $v$. The proof is basically the same as we did in \eqref{2.25}--\eqref{2.28} and the main differences one should notice are terms $K_3$ and $J_2$--$J_4$. For $K_3$, 
\begin{equation}\label{4.26}
\begin{aligned}
K_3&=-\int \frac{1}{2}\mu(\rho)_t |\curle v|^2\\
&=-\frac{1}{2}\int_{\partial}(n\times v)\cdot \curle v\mu(\rho)_t- \frac{1}{2}\int \nabla\mu(\rho)_t\times \curle v\cdot v-\frac{1}{2}\int \mu(\rho)_t\Delta v\cdot v\\
&=-\frac{1}{2}\int_\partial \rho\mu'(\rho)(\log\rho)_tv\cdot B\cdot v- \frac{1}{2}\int \rho\mu'(\rho)\nabla(\log\rho)_t\times \curle v\cdot v\\
&\quad- \frac{1}{2}\int \left[\rho\mu'(\rho)+\rho^2\mu''(\rho)\right](\log\rho)_t\nabla\log\rho\times \curle v\cdot v-\frac{1}{2}\int \mu(\rho)_t\Delta v\cdot v\\
&\leq |K_2|+C\left(\norm{\nabla(\log\rho)_t}_{L^2}+\norm{\nabla\rho}_{L^r}\norm{(\log\rho)_t}_{L^{\frac{2r}{r-2}}}\right)\norm{v}_{L^r}\norm{\nabla v}_{L^{\frac{2r}{r-2}}}\\
&\quad+C\norm{v}_{L^r}\norm{(\log\rho)_t}_{L^{\frac{2r}{r-2}}}\norm{\Delta v}_{L^2}\\
&\leq |K_2|+C_\varepsilon\left(\norm{v}_{L^r}^s+\norm{\nabla\rho}_{L^r}^s+1\right)\left(\norm{\nabla v}_{L^2}^2+\norm{(\log\rho)_t}_{L^2}^2\right)\\
&\quad+\varepsilon\left(\norm{\nabla(\log\rho)_t}_{L^2}^2+\norm{v}_{H^2}^2\right),
\end{aligned}
\end{equation}
while, for $J_2$--$J_4$, using the relation
$$\nabla^2\rho^{-1}=\frac{1}{\rho^2}\nabla^2\rho-\frac{2}{\rho}\nabla^2\log\rho$$
and Lemma \ref{Lemma221}, we have
\begin{equation}\label{4.27}
\begin{aligned}
J_2&=\int c_0\dive{\left[2\mu(\rho)\nabla^2\rho^{-1}\right]}\cdot v_t\\
&\leq C\norm{\nabla\rho}_{L^r}\normf{\nabla^2\rho^{-1}}_{L^{\frac{2r}{r-2}}}\norm{v_t}_{L^2}+C\norm{\nabla\Delta\rho^{-1}}_{L^2}\norm{v_t}_{L^2}\\
&\leq C_{\varepsilon_1}\left(\norm{\nabla\rho}_{L^r}^s+1\right)\left(\normf{\Delta\rho}_{L^2}^2+\normf{\Delta\log\rho}_{L^2}^2\right)\\
&\quad+C_{\varepsilon_1}\left(\normf{\nabla\Delta\rho}_{L^2}^2+\normf{\nabla\Delta\log\rho}_{L^2}^2\right)+\varepsilon_1\norm{v_t}_{L^2}^2
\end{aligned}
\end{equation}
\begin{equation}\label{4.28}
\begin{aligned}
J_3&=-\int c_0 \dive{\left(\rho v\otimes\nabla\rho^{-1}\right)}\cdot v_t=\int c_0 \dive{\left( v\otimes\nabla\log\rho\right)}\cdot v_t\\
&\leq C\norm{\nabla\rho}_{L^r}\norm{\nabla v}_{L^{\frac{2r}{r-2}}}\norm{v_t}_{L^2}+C\norm{v}_{L^r}\normf{\nabla^2\log\rho}_{L^{\frac{2r}{r-2}}}\norm{v_t}_{L^2}\\
&\leq C_{\varepsilon_2}\left(\norm{\nabla\rho}_{L^r}^s+\norm{v}_{L^r}^s+1\right)\left(\norm{\Delta\log\rho}_{L^2}^2+\norm{\nabla v}_{L^2}^2\right)\\
&\quad+\varepsilon_2\left(\norm{\nabla\Delta\log\rho}_{L^2}^2+\norm{v}_{H^2}^2+\norm{v_t}_{L^2}^2\right)
\end{aligned}
\end{equation}
\begin{equation}\label{4.29}
\begin{aligned}
J_4&=-\int c_0^2 \dive{\left(\rho \nabla\rho^{-1}\otimes\nabla\rho^{-1}\right)}\cdot v_t=\int c_0^2 \dive{\left(\nabla\log\rho\otimes\nabla\rho^{-1}\right)}\cdot v_t\\
&\leq C\norm{\nabla\rho}_{L^r}\left(\normf{\nabla^2\log\rho}_{L^{\frac{2r}{r-2}}}+\normf{\nabla^2\rho^{-1}}_{L^{\frac{2r}{r-2}}}\right)\norm{v_t}_{L^2}\\
&\leq C_{\varepsilon_3}\left(\norm{\nabla\rho}_{L^r}^s+1\right)\left(\norm{\Delta\log\rho}_{L^2}^2+\norm{\Delta\rho}_{L^2}^2\right)\\
&\quad+\varepsilon_3\left(\norm{\nabla\Delta\log\rho}_{L^2}^2+\norm{\nabla\Delta\rho}_{L^2}^2+\norm{v_t}_{L^2}^2\right)\\
\end{aligned}
\end{equation}
Therefore, modifying the corresponding norms of $(v,\nabla\rho)$ from \eqref{2.25}--\eqref{2.28} into the $L^r$-norms, alonging with \eqref{4.26}--\eqref{4.29}, we have
\begin{equation}\label{4.30}
\begin{aligned}
&\left(\int_\partial \mu(\rho)v\cdot B\cdot v+\int \mu(\rho)|\curle v|^2\right)_t +\nu\norm{v_t}_{L^2}^2+M'(t)\\
&\leq C_\varepsilon\left(\norm{\nabla\rho}_{L^3}^4+\norm{\nabla\rho}_{L^r}^s+\norm{v}_{L^3}^4+\norm{v}_{L^r}^s+1\right)\left(\norm{\nabla v}_{L^2}^2+\norm{\Delta\log\rho}_{L^2}^2+\norm{\Delta\rho}_{L^2}^2\right)\\
&\quad+C_{\varepsilon}\left(\norm{\nabla\Delta\log\rho}_{L^2}^2+\norm{\nabla\Delta\rho}^2_{L^2}\right)+\varepsilon\left(\norm{v}_{H^2}^2+\norm{\nabla(\log\rho)_t}_{L^2}^2\right).
\end{aligned}
\end{equation}
For the sake of simplicity, as we have explained in \eqref{327} and \eqref{341}, we can rewrite \eqref{4.30} into
\begin{equation}\label{4.30}
\begin{aligned}
\left(\norm{\nabla v}_{L^2}^2\right)_t +\nu\norm{v_t}_{L^2}^2&\leq C_\varepsilon\left[\cI(t)+1\right]\left(\norm{\nabla v}_{L^2}^2+\norm{\Delta\log\rho}_{L^2}^2+\norm{\Delta\rho}_{L^2}^2\right)\\
&\quad+C_{\varepsilon}\left(\norm{\nabla\Delta\log\rho}_{L^2}^2+\norm{\nabla\Delta\rho}^2_{L^2}\right)+\varepsilon\left(\norm{v}_{H^2}^2+\norm{\nabla(\log\rho)_t}_{L^2}^2\right),
\end{aligned}
\end{equation}
where $\cI(t)$ is an integrable function over $(0,T^*)$. 

For $H^2$-norm of $v$, analoging with \eqref{328}--\eqref{331} and applying Lemma \ref{Lemma221}, one has
\begin{equation}\label{44.32}
\begin{aligned}
\norm{F}^2_{L^2}&\leq C_{\varepsilon}\left(\norm{\nabla\rho}_{L^r}^s+\norm{v}_{L^r}^s+1\right)\left(\norm{\Delta\log\rho}_{L^2}^2+\norm{\Delta\rho}_{L^2}^2+\norm{\nabla v}_{L^2}^2\right)\\
&\quad+C\left(\norm{v_t}^2_{L^2}+\norm{\nabla(\log\rho)_t}^2_{L^2}\right)+\varepsilon\left(\norm{v}_{H^2}^2+\norm{\nabla\Delta\log\rho}_{L^2}^2+\norm{\nabla\Delta\rho}^2_{L^2}\right)\\
\norm{\Phi}^2_{H^1}&\leq C\left(\norm{\nabla v}_{L^2}^2+\norm{\Delta\rho}_{L^2}^2+\norm{\Delta\log\rho}_{L^2}^2\right).
\end{aligned}
\end{equation}
where $\Phi:=-B\cdot[v+c_0\nabla\rho^{-1}]$. Thus, from Lemma \ref{lemma2.8}, \eqref{2.15}, we have
\begin{equation}\label{4.33}
\begin{aligned}
\norm{v}_{H^2}^2+\norm{p}_{H^1}^2&\leq C_{\varepsilon}\left(\norm{\nabla\rho}_{L^r}^s+\norm{v}_{L^r}^s+1\right)\left(\norm{\Delta\log\rho}_{L^2}^2+\norm{\Delta\rho}_{L^2}^2+\norm{\nabla v}_{L^2}^2\right)\\
&\quad+C\left(\norm{v_t}^2_{L^2}+\norm{\nabla(\log\rho)_t}^2_{L^2}\right)+\varepsilon\left(\norm{\nabla\Delta\log\rho}_{L^2}^2+\norm{\nabla\Delta\rho}^2_{L^2}\right),
\end{aligned}
\end{equation}
alonging with \eqref{4.30} gives
\begin{equation}\label{4.34}
\begin{aligned}
\left(\norm{\nabla v}_{L^2}^2\right)_t +\frac{\varepsilon}{2C}\norm{v}_{H^2}^2+\frac{\nu}{2}\norm{v_t}_{L^2}^2&\leq C_\varepsilon\left[\cI(t)+1\right]\left(\norm{\nabla v}_{L^2}^2+\norm{\Delta\log\rho}_{L^2}^2+\norm{\Delta\rho}_{L^2}^2\right)\\
&\quad+C_{\varepsilon}\left(\norm{\nabla\Delta\log\rho}_{L^2}^2+\norm{\nabla\Delta\rho}^2_{L^2}\right)\\
&\quad+\varepsilon\norm{\nabla(\log\rho)_t}_{L^2}^2.
\end{aligned}
\end{equation}
Thus, combining \eqref{4.22}, \eqref{4.25}, \eqref{4.33} and \eqref{4.34} by using the similar approach from \eqref{332}--\eqref{333}, then applying the Gr$\mathrm{\ddot{o}}$nwall's inequality, we deduce the estimate \eqref{4.11}.
\end{proof}

\begin{Remark}
From the proof above, one should notice that, it is the convection term $\rho u\cdot \nabla u$ that restricts us to use the Serrin's condition of $v$. In fact, we can directly use the the bound $u\in L^s(0,T;L^r)$ in \eqref{44.6} to get the lower bounds for $\log\rho$ (see also Lemma \ref{lemma46}), but, in order to show this point, we insist to only use $\nabla\rho\in L^s(0,T;L^r)$.
\end{Remark}

Now, we turn back to prove Proposition \ref{prop4.1} for $(\rho,u)$ satisfying \eqref{equation1.6}.
\begin{proof}[Proof of Proposition \ref{prop4.1}]
Combining Lemma \ref{lemma44}--\ref{lemma45}, we can get Proposition \ref{prop4.1}. The only point one should notice is that 
\begin{equation*}
\begin{aligned}
\norm{\nabla\rho_t}_{L^2}&\leq C\left(\norm{\rho_t\nabla\rho}_{L^2}+\norm{\nabla(\log\rho)_t}_{L^2}\right)\\
&\leq C\left(\norm{\Delta\rho}^2_{L^2}\norm{\rho_t}_{L^2}+\norm{\nabla(\log\rho)_t}_{L^2}\right)+\frac{1}{2}\norm{\nabla\rho_t}_{L^2},
\end{aligned}
\end{equation*}
that is,
\begin{equation*}
\int_0^T\norm{\nabla\rho_t}^2_{L^2}\,dt\leq C\left(\sup_{t\in[0,T]}\norm{\Delta\rho}^2_{L^2}\sup_{t\in[0,T]}\norm{\rho_t}^2_{L^2}\int_0^T\norm{\Delta\rho}^2_{L^2}\,dt+\int_0^T\norm{\nabla(\log\rho)_t}^2_{L^2}\,dt\right)\leq \tilde C.
\end{equation*}
\end{proof}

\subsection{Case for $(\rho,u)$ satisfying \eqref{equation1.7}}

We basically follow the proof in subsection \ref{11}. Since the nonlinear term $|\nabla\rho|^2$, one still has to estimate for $\rho$ together with $\log\rho$. In case of use, we colloect some bounds from \eqref{1.23}
\begin{equation}\label{Q1}
\norm{\nabla Q}_{L^2}\leq C\left(\norm{\Delta\rho}_{L^2}+\norm{\nabla\rho}_{L^r}\norm{\nabla\rho}_{L^{\frac{2r}{r-2}}}\right)\leq C\left[\norm{\Delta\rho}_{L^2}+\left(\norm{\nabla\rho}_{L^r}^s+1\right)\norm{\nabla\rho}_{L^2}^2\right],
\end{equation}
\begin{equation}\label{Q2}
\begin{aligned}
\norm{Q_t}_{L^2}&\leq C\left(\norm{\nabla(\log\rho)_t}_{L^2}+\norm{\nabla\rho}_{L^r}\norm{\rho_t}_{L^{\frac{2r}{r-2}}}\right)\\
&\leq C\left[\norm{\nabla(\log\rho)_t}_{L^2}+\left(\norm{\nabla\rho}_{L^r}^s+1\right)\norm{(\log\rho)_t}_{L^2}^2\right].
\end{aligned}
\end{equation}

First, we give a lemma for the lower order bounds of $\rho$.

\begin{Lemma}\label{lemma46}
Suppose that $(\rho,u)$ satisfies the condition \eqref{equation1.7}. Then, for all $T\in (0,T^*)$, Lemma \ref{lemma4.2} holds and 
\begin{equation}
\sup_{t\in[0,T]}\left(\norm{\nabla\rho}_{L^2}^2+\norm{\nabla\log\rho}_{L^2}^2\right)+\int_0^T\left(\norm{\Delta\rho}_{L^2}^2+\norm{\Delta\log\rho}_{L^2}^2\right)\,dt\leq \tilde C.
\end{equation}
\end{Lemma}
\begin{proof}
The estimates for $\rho$ and $\log\rho$ come from \eqref{eq3.7} and \eqref{44.6}, respectively, and we only give the proof for $\log\rho$ here, since another one can be proved similarly. From \eqref{44.6}, we have
\begin{equation}\label{44.40}
\begin{aligned}
\left(\frac{1}{2}\int |\nabla\log\rho|^2\right)_t+\int c_0\rho^{-1}|\Delta\log\rho|^2&=\int (v\cdot \nabla\log\rho)\Delta\log\rho\\
&\leq\norm{v}_{L^{r}} \norm{\nabla\log\rho}_{L^\frac{2r}{r-2}}\norm{\Delta\log\rho}_{L^2}\\
&\leq C_\varepsilon\left(\norm{v}_{L^r}^s+1\right)\norm{\nabla\log\rho}^2_{L^2}+\varepsilon\norm{\Delta\log\rho}_{L^2}^2\\
&\leq C_\varepsilon\left(\norm{u}_{L^r}^s+1\right)\norm{\nabla\log\rho}^2_{L^2}+\varepsilon\norm{\Delta\log\rho}_{L^2}^2,
\end{aligned}
\end{equation}
then, applying \eqref{4.1} and Gr$\mathrm{\ddot{o}}$nwall's inequality for \eqref{44.40}, we conclude the proof.
\end{proof}

\begin{Lemma}
Suppose that $(\rho,u)$ satisfies the condition \eqref{equation1.7}. Then, 
\begin{equation}\label{con}
\sup_{t\in [0,T]}\tilde\cF(t)+\int_0^T\left(\tilde\cG(t)+\norm{\pi}_{H^1}\right)\,dt\leq \tilde C,
\end{equation}
where
\begin{equation*}
\begin{gathered}
\tilde\cF(t):=\normf{u}_{H^1}^2+\norm{\Delta\rho}_{L^2}^2+\norm{\Delta\log\rho}_{L^2}^2+\norm{(\log\rho)_t}_{L^2}^2,\\
\tilde\cG(t):= \norm{\nabla\Delta\rho}_{L^2}^2+ \norm{\nabla\Delta\log\rho}_{L^2}^2+\norm{u_t}_{L^2}^2+\norm{\Delta u}_{L^2}^2+\norm{\nabla(\log\rho)_t}_{L^2}^2
\end{gathered}
\end{equation*}
\end{Lemma}
\begin{proof}
On the one hand, We follow the proof from \eqref{44.12} to \eqref{4.22} and replace all $\norm{v}_{L^3}^4$ by $\norm{v}_{L^r}^s$ via Lemma \ref{Lemma221} to obtain that
\begin{equation}\label{44.39}
\begin{aligned}
&\left(\norm{\Delta\rho}_{L^2}^2+\norm{\Delta\log\rho}_{L^2}^2\right)_t+\nu\left(\norm{\nabla\Delta\rho}_{L^2}^2+\norm{\nabla\Delta\log\rho}_{L^2}^2\right)\\
&\leq C\left(\norm{\nabla\rho}_{L^r}^s+\norm{v}_{L^r}^s+1\right)\left(\norm{\Delta\rho}_{L^2}^2+\norm{\Delta\log\rho}_{L^2}^2+\norm{\nabla v}_{L^2}^2\right)+\varepsilon\norm{v}_{H^2}^2\\
&\leq C\left(\norm{u}_{L^r}^s+1\right)\tilde\cF(t)+\varepsilon\norm{v}_{H^2}^2,
\end{aligned}
\end{equation}
On the other hand, we still have \eqref{4.25}, that is,
\begin{equation}\label{44.40}
\begin{aligned}
\left(\norm{(\log\rho)_t}_{L^2}^2\right)_t+\nu\norm{\nabla(\log\rho)_t}_{L^2}^2&\leq C_\varepsilon(\norm{\nabla\rho}_{L^r}^s+1)\norm{(\log\rho)_t}_{L^2}^2+\varepsilon\norm{v_t}_{L^2}^2\\
&\leq C_\varepsilon(\norm{u}_{L^r}^s+1)\tilde\cF(t)+\varepsilon\norm{u_t}_{L^2}^2.
\end{aligned}
\end{equation}
Here, we have used the fact that
$$\norm{\nabla v}_{L^2}^2\leq C\left(\norm{\nabla u}_{L^2}^2+\normf{\nabla^2\rho^{-1}}_{L^2}^2\right)\leq C\left(\norm{\nabla u}_{L^2}^2+\normf{\Delta\rho}_{L^2}^2+\normf{\Delta\log\rho}_{L^2}^2\right),$$
$$\norm{v_t}_{L^2}^2\leq C\left(\norm{u_t}_{L^2}^2+\normf{\nabla\rho^{-1}_t}_{L^2}^2\right)\leq C\left(\norm{u_t}_{L^2}^2+\normf{\nabla(\log\rho)_t}_{L^2}^2+\normf{\nabla\rho}_{L^r}^2\normf{(\log\rho)_t}_{L^{\frac{2r}{r-2}}}^2\right).$$

For $u$, similar with the proof in subsection \ref{11}, we apply the Serrin's condition \eqref{4.1} on \eqref{346}--\eqref{351} and use \eqref{Q1}--\eqref{Q2}, we can derive that 
\begin{equation}\label{44.41}
\begin{aligned}
&\left(\norm{\sqrt\rho u}_{L^2}^2\right)_t+\nu\norm{\nabla u}_{L^2}^2\\
&\leq C_\varepsilon\left(\norm{u}_{L^r}^s+1\right)\left(\norm{\sqrt\rho u}_{L^2}^2+\norm{\nabla\rho}_{L^2}^2\right)+C_\varepsilon\left(\norm{\nabla\rho}_{L^2}^2+\norm{\Delta\rho}_{L^2}^2\right)+\varepsilon\norm{u_t}_{L^2}^2\\
&\leq C_\varepsilon\left(\norm{u}_{L^r}^s+1\right)\tilde\cF(t)+C_\varepsilon\left(\norm{\nabla\rho}_{L^2}^2+\norm{\Delta\rho}_{L^2}^2\right)+\varepsilon\norm{u_t}_{L^2}^2,
\end{aligned}
\end{equation}
and
\begin{equation}\label{44.42}
\begin{aligned}
\left(\normf{\sqrt{\mu(\rho)}|D(u)|}_{L^2}^2\right)_t+\nu\norm{u_t}_{L^2}^2&\leq C_\varepsilon\left(\norm{u}_{L^r}^s+1\right)\tilde\cF(t)+C_\varepsilon\norm{\nabla(\log\rho)_t}_{L^2}^2+\varepsilon\norm{\Delta u}_{L^2}^2,
\end{aligned}
\end{equation}
where the only term we need concern is
\begin{equation*}
N_3=\int\mu(\rho)_t|D(u)|^2\quad\text{ in }\eqref{350}.
\end{equation*}
However, this term can be computed by integrating by parts,
\begin{equation*}
\begin{aligned}
N_3&=\int\mu(\rho)_t|D(u)|^2\\
&=-\int \nabla\mu(\rho)_t\cdot D(u)\cdot u-\int \frac{1}{2}\mu(\rho)_t\Delta u\cdot u\\
&=-\int \rho\mu'(\rho)\nabla(\log\rho)_t\cdot D(u)\cdot u-\int (\log\rho)_t \left(\rho\mu'(\rho)\right)'\nabla\rho\cdot D(u)\cdot u\\
&\quad-\int \frac{1}{2}\rho\mu'(\rho)(\log\rho)_t\Delta u\cdot u\\
&\leq C\norm{u}_{L^r}\norm{\nabla u}_{L^{\frac{2r}{r-2}}}\norm{\nabla(\log\rho)_t}_{L^2}+C\norm{\nabla\rho}_{L^r}\norm{(\log\rho)_t}_{L^{\frac{2r}{r-2}}}\norm{u}_{L^r}\norm{\nabla u}_{L^{\frac{2r}{r-2}}}\\
&\quad+C\norm{u}_{L^r}\norm{(\log\rho)_t}_{L^{\frac{2r}{r-2}}}\norm{\Delta u}_{L^2}\\
&\leq C_\varepsilon\left(\norm{u}_{L^r}^s+1\right)\left(\norm{\nabla u}_{L^2}^2+\norm{(\log\rho)_t}_{L^2}^2\right)+\varepsilon\left(\norm{\nabla(\log\rho)_t}_{L^2}^2+\norm{\Delta u}_{L^2}^2\right)\\
&\leq C_\varepsilon\left(\norm{u}_{L^r}^s+1\right)\tilde\cF(t)+\varepsilon\left(\norm{\nabla(\log\rho)_t}_{L^2}^2+\norm{v}_{H^2}^2\right),
\end{aligned}
\end{equation*}
where we have used 
$$
\begin{aligned}
\norm{\Delta u}^2_{L^2}&\leq C\left(\norm{v}_{H^2}^2+\norm{\nabla\Delta\rho^{-1}}_{L^2}^2\right)\\
&\leq C\left(\norm{v}_{H^2}^2+\norm{\nabla\Delta\rho}^2_{L^2}+\norm{\nabla\Delta\log\rho}^2_{L^2}\right)\\
&\quad+C\norm{\nabla\rho}^2_{L^r}\left(\norm{\Delta\rho}_{L^{\frac{2r}{r-2}}}^2+\norm{\Delta\log\rho}_{L^{\frac{2r}{r-2}}}^2\right)
\end{aligned}
$$

To estimate $\Delta u$, we apply Lemma \ref{Lemma221}, \ref{lemma2.7} and \ref{lemma2.8} on \eqref{328} and, then, use \eqref{44.32}--\eqref{4.33} with $\Phi=-c_0\nabla\rho^{-1}$ and
$$
\begin{aligned}
\norm{\Phi}_{H^2}&\leq C\norm{\nabla\Delta\rho^{-1}}_{L^2}\\
&\leq C\left(\norm{\nabla\Delta\rho}_{L^2}+\norm{\nabla\Delta\log\rho}_{L^2}\right)\\
&\quad+C\norm{\nabla\rho}_{L^r}\left(\norm{\Delta\rho}_{L^{\frac{2r}{r-2}}}+\norm{\Delta\log\rho}_{L^{\frac{2r}{r-2}}}\right)
\end{aligned}
$$
 to deduce that 
 \begin{equation}\label{44.43}
\begin{aligned}
\norm{v}_{H^2}^2+\norm{\pi}_{H^1}^2&\leq C\left(\norm{u}_{L^r}^s+1\right)\left(\norm{\Delta\log\rho}_{L^2}^2+\norm{\Delta\rho}_{L^2}^2+\norm{\nabla v}_{L^2}^2\right)\\
&\quad+C\left(\norm{v_t}^2_{L^2}+\norm{\nabla(\log\rho)_t}^2_{L^2}+\norm{\nabla\Delta\log\rho}_{L^2}^2+\norm{\nabla\Delta\rho}^2_{L^2}\right)\\
&\leq C\left(\norm{u}_{L^r}^s+1\right)\tilde\cF(t)\\
&\quad+C\left(\norm{u_t}^2_{L^2}+\norm{\nabla(\log\rho)_t}^2_{L^2}+\norm{\nabla\Delta\log\rho}_{L^2}^2+\norm{\nabla\Delta\rho}^2_{L^2}\right).
\end{aligned}
\end{equation}

Now, collecting the bounds \eqref{44.39}--\eqref{44.43} and following the proof from \eqref{3354} to \eqref{3361}, one has 
\begin{equation}\label{44.44}
\begin{aligned}
\tilde\cF'(t)+\nu\tilde\cG(t)\leq C(\norm{u}_{L^r}^s+1)\tilde\cF(t)+C\left(\norm{\nabla\rho}_{L^2}^2+\norm{\Delta\rho}_{L^2}^2\right).
\end{aligned}
\end{equation}
Applying the Gr$\mathrm{\ddot{o}}$nwall's inequality and Lemma \ref{lemma46} on \eqref{44.44} and, then, turning back to \eqref{44.43}, we obtain \eqref{con}.
\end{proof} 

The proof of Proposition \ref{prop4.1} is same as that at the end of subsection \ref{11}, we omit it and left it to readers.

\subsection{Proof of Theorem \ref{Theorem1.3}}

Since we have Proposition \ref{prop4.1} and the constant $\tilde C$ is independent with $T\in (0,T^*)$. Thus, we can let $t\to T^*$ and consider $(\rho,u)(x,T^*)$ as the initial data. Then, following the proof in subsection \ref{P12}, we can deduce the violation of the maximality of $T^*$. Therefore, we complete the proof for Theorem \ref{Theorem1.3}.

\bibliographystyle{abbrv}
\bibliography{reference2}

\end{document}